\documentclass[12pt,a4paper,reqno]{amsart}
\usepackage[english]{babel}
\usepackage[applemac]{inputenc}
\usepackage[T1]{fontenc}
\usepackage{palatino}
\usepackage{amsmath}
\usepackage{amssymb}
\usepackage{amsthm}
\usepackage{amsfonts}
\usepackage{graphicx}
\usepackage[colorlinks = true, citecolor = black]{hyperref}
\author{Katrin F\"assler and Tuomas Orponen}
\thanks{K.F. was supported by the Swiss National Science Foundation and by the Academy of Finland, project number 252293. T.O. was partially supported by the Academy of Finland, grant 133264 "Stochastic and harmonic analysis, interactions and applications".}

\title{On restricted families of projections in $\R^{3}$}
\address{Department of Mathematics and Statistics, University of Helsinki, P.O.B. 68, FI-00014 Helsinki, Finland}
\subjclass[2010]{28A80 (Primary); 28A78, 37C45 (Secondary).}
\email{katrin.fassler@helsinki.fi \\ tuomas.orponen@helsinki.fi}

\newcommand{\R}{\mathbb{R}}
\newcommand{\N}{\mathbb{N}}

\newcommand{\cH}{\mathcal{H}}
\newcommand{\Bd}{\overline{\dim}_{\mathrm{B}}}
\newcommand{\cD}{\mathcal{D}}

\newcommand{\cE}{\mathcal{E}}

\newcommand{\Hd}{\dim_{\mathrm{H}}}
\newcommand{\Pd}{\dim_{\mathrm{p}}}

\newcommand{\spa}{\operatorname{span}}
\newcommand{\diam}{\operatorname{diam}}

\numberwithin{equation}{section}

\theoremstyle{plain}
\newtheorem{thm}[equation]{Theorem}
\newtheorem{lemma}[equation]{Lemma}

\newtheorem{cor}[equation]{Corollary}
\newtheorem{proposition}[equation]{Proposition}

\newtheorem{conjecture}[equation]{Conjecture}

\theoremstyle{definition}
\newtheorem{definition}[equation]{Definition}

\theoremstyle{remark}
\newtheorem{remark}[equation]{Remark}

\begin{document}

\begin{abstract} We study projections onto non-degenerate one-dimensional families of lines and planes in $\R^{3}$. Using the classical potential theoretic approach of R. Kaufman, one can show that the Hausdorff dimension of at most $1/2$-dimensional sets $B \subset \R^{3}$ is typically preserved under one-dimensional families of projections onto lines. We improve the result by an $\varepsilon$, proving that if $\Hd B = s > 1/2$, then the packing dimension of the projections is almost surely at least $\sigma(s) > 1/2$. For projections onto planes, we obtain a similar bound, with the threshold $1/2$ replaced by $1$. In the special case of self-similar sets $K \subset \R^{3}$ without rotations, we obtain a full Marstrand type projection theorem for one-parameter families of projections onto lines. The $\Hd K \leq 1$ case of the result follows from recent work of M. Hochman, but the $\Hd K > 1$ part is new: with this assumption, we prove that the projections have positive length almost surely.

\end{abstract}

\maketitle

\tableofcontents

\section{Introduction}

The study of orthogonal projections has a long history in the field of geometric measure theory. The foundations were laid in the late thirties by A.S. Besicovitch, who established that the structural properties of sets of Hausdorff dimension one in $\R^{2}$ are reflected in the size of their orthogonal projections into lines through the origin. For planar sets of arbitrary dimension, a major breakthrough appeared some 15 years later, in the 1954 paper \cite{Mar} by J.M. Marstrand. He demonstrated that for sets $B \subset \R^{2}$ with dimension at most one, almost all projections have the same dimension, while the assumption $\Hd B > 1$ guarantees that almost all projections have positive length. In the present article, we aim for results of this nature in $\R^{3}$.

In $\R^{d}$, $d \geq 3$, there are at least $(d - 1)$ natural generalisations of Marstrand's theorem. Namely, one may pick $k \in \{1,2,\ldots,d - 1\}$ and start asking questions about orthogonal projections onto $k$-dimensional subspaces. In this situation, it turns out that the dimension of at most $k$-dimensional sets $B \subset \R^{d}$ is preserved under almost all projections, while $\Hd B > k$ suffices for positive $k$-dimensional measure -- but, again, only for almost all projections. The generalisation, published in 1975, is due to P. Mattila \cite{Mat}.

The results are complete, save for the word "almost". A moment's thought reveals that the word may not be entirely omitted, as it is easy to come up with examples of sets $B \subset \R^{d}$ (line segments, for instance) for which the dimension of projections is less than $\Hd B$ for a few \emph{exceptional} subspaces. This does not mean that results sharper than the ones by Marstrand and Mattila are not possible -- and indeed they are: numerous such improvements have appeared since 1954. A particularly elegant one is due to R. Kaufman \cite{Ka} from 1968: given $0 < s < 1$, a family $\mathcal{L}$ of lines through the origin in $\R^{2}$ such that $\Hd \{L \cap S^{1} : L \in \mathcal{L}\} = s$, and a set $B \subset \R^{2}$ with $\Hd B < s$, one may always find a line $L \in \mathcal{L}$ such that the projection of $B$ into $L$ has dimension $\Hd B$. Moreover, the result is sharp in the sense that it may fail if $\Hd B = s$. This was shown by Kaufman and Mattila \cite{KM} in 1975: for any $0 < s < 1$, they managed to construct an $s$-dimensional set $B \subset \R^{2}$, the dimension of the projections of which drops strictly below $s$ for a certain $s$-dimensional family $\mathcal{L} = \mathcal{L}_{B}$ of lines through the origin.

To sum up, in casual terms, the results mentioned so far, it has been known for quite some time that orthogonal projections preserve dimension almost surely, but the family of exceptional projections can be fairly large. A question that remains, to date, largely unanswered, can be phrased as follows: What is the structure of exceptional sets (of subspaces)? For instance, the construction of Kaufman and Mattila from 1975 essentially relies on the fact that both $B$ and $\mathcal{L}_{B}$ can be chosen freely, and so as to play well together: it is far from clear that a less carefully chosen $\mathcal{L}$ would consist entirely of exceptional lines for \textbf{any} $s$-dimensional set $B$. 

In $\R^{2}$, the question is wide open, but in $\R^{3}$ some understanding is emerging. For those interested in a philosophical reason for how $\R^{3}$ can possibly be easier than $\R^{2}$, it is that the full families of projections onto one- and two-dimensional subspaces in $\R^{3}$ are two-dimensional -- they can both be naturally parametrised by the unit sphere $S^{2}$ -- so one can ask non-trivial questions about one-dimensional subfamilies. In the plane, however, the full family of lines is only one-dimensional, so the interesting questions necessarily concern subfamilies of fractional dimension, and, at the moment the research seems to be even devoid of plausible conjectures.

For the rest of the introduction -- and indeed the paper -- we will be concerned with projections in $\R^{3}$ onto one- and two-dimensional subspaces. More precisely, we are interested in smooth one-dimensional subfamilies of the full (two-dimensional) families. For the moment, let us fix such a family $\mathcal{L}$ of one-dimensional subspaces (essentially the same considerations are relevant for families of two-dimensional subspaces, so we will not grant them a separate treatment in this informal discussion). The smoothness is quantified here by parametrising the family $\mathcal{L}$ by a smooth path $\gamma \colon U \to S^{2}$ (here $U \subset \R$ is an open interval) so that
\begin{displaymath} \mathcal{L} = (\ell_{\theta})_{\theta \in U} := \{\operatorname{span}(\gamma(\theta)) : \theta \in U\}. \end{displaymath}
The union of the lines in $\mathcal{L}$ forms a surface $S$, and it turns out that the curvature of $S$ plays a crucial role in our investigation. The necessity of curvature for non-trivial results is easy to see: if $S$ is completely flat, that is, contained in a single two-dimensional subspace $V$, then so are all the lines in $\mathcal{L}$, and the projection of the one-dimensional set $B = V^{\perp}$ into each line in $\mathcal{L}$ is the singleton $\{0\}$. Then $\mathcal{L}$ consists entirely of exceptional lines with respect to $B$, in a very strong sense. More generally, if $S$ is contained in a countable union of two-dimensional subspaces, then $\mathcal{L}$ can be shown to be entirely exceptional for some one-dimensional set $B \subset \R^{3}$ -- again in the sense that the dimension of the projection of $B$ into $L$ is zero for every line $L \in \mathcal{L}$. 

For one-dimensional families of planes, the "non-curved" situation is slightly more subtle, and for low-dimensional families of $k$-dimensional subspaces in $\R^{d}$, the subtlety increases still: nevertheless, the best possible projection results \textbf{without} curvature conditions are known for all pairs $k < d$, due to the work of M. J\"arvenp\"a\"a, E. J\"arvenp\"a\"a, T. Keleti, F. Ledrappier and M. Leikas, see \cite{JJK} and \cite{JJLL}. At any rate, if we are interested in progress towards a Marstrand type theorem for $\mathcal{L}$, or for one-dimensional families of two-dimensional subspaces, we need to assume some curvature. Stated in terms of the parametrising path $\gamma$, the condition used in the present paper reads as follows:

\begin{definition}[Non-degenerate families]\label{d:non-deg} Let $U \subset \R$ be an open interval, and let $\gamma: U \to S^2$ be a $\mathcal{C}^{3}$-curve on the unit sphere in $\mathbb{R}^3$ satisfying the condition
\begin{equation}\label{eq:curve_cond} \spa(\{\gamma(\theta),\dot{\gamma}(\theta),\ddot{\gamma}(\theta)\}) = \R^{3}, \qquad \theta \in U. \end{equation}
To each point $\gamma(\theta)$, $\theta \in U$, we assign the line $\ell_{\theta} = \spa(\gamma(\theta))$. Any family of lines $(\ell_{\theta})_{\theta \in U}$ so obtained is called a \emph{non-degenerate family of lines}. The orthogonal complements $V_{\theta} := \ell_{\theta}^{\perp}$ form a one-dimensional family of planes. Any family of planes $(V_{\theta})_{\theta \in U}$ so obtained is called a \emph{non-degenerate family of planes}.
\end{definition}

As we have indicated, we are interested in projections corresponding to non-degenerate families of lines and planes. For these, we use the following notation:
\begin{definition}[Projections $\rho_{\theta}$ and $\pi_{\theta}$]\label{projections} If $(\ell_{\theta})_{\theta \in U}$ is a non-degenerate family of lines associated with the curve $\gamma \colon U \to S^{2}$, we write $\rho_{\theta} \colon \R^{3} \to \R$ for the orthogonal projection
\begin{displaymath} \rho_{\theta}(x) = \gamma(\theta) \cdot x.  \end{displaymath}
Thus, we interpret the projection onto the line $\ell_{\theta}$ spanned by $\gamma({\theta})$ as a subset of $\R$. Given a non-degenerate family of planes $(V_{\theta})_{\theta \in U}$, we denote by $\pi_{\theta} \colon \R^{3} \to \R^{2}$ the orthogonal projection onto the plane $V_{\theta}$, identified with $\R^{2}$. Projection families of the form $(\rho_{\theta})_{\theta \in U}$ and $(\pi_{\theta})_{\theta \in U}$ will be referred to as \emph{non-degenerate families of projections}.
\end{definition}

\begin{remark} The terminology "non-degenerate families of projections" is also used in the papers \cite{JJK} and \cite{JJLL} mentioned above, where no curvature assumptions are imposed. So, our definition is more restrictive, despite the common name. \end{remark}

\subsection{Dimension estimates for general sets}

Unless otherwise stated, $(\rho_{\theta})_{\theta \in U}$ and $(\pi_{\theta})_{\theta \in U}$ will always stand for non-degenerate families of projections onto lines and planes, respectively. When $\Hd B$ lies on certain intervals, dimension conservation for non-degenerate families of projections can be proven directly using the classical `potential theoretic' method pioneered by R. Kaufman in \cite{Ka}. These bounds are the content of the following proposition.
\begin{proposition}\label{prop1} Let $B \subset \R^{3}$ be an analytic set.
\begin{itemize}
\item[(a)] If $\Hd B \leq 1/2$, then $\Hd \rho_{\theta}(B) = \Hd B$ almost surely.
\item[(b)] If $\Hd B \leq 1$, then $\Hd \pi_{\theta}(B) = \Hd B$ almost surely.
\end{itemize}
\end{proposition}

Part (b) follows from \cite[Proposition 3.2]{JJLL}. The proof of part (a) is standard: we include it mainly to identify the `enemy' against which we have to combat in order to obtain an improvement, but also because the proof contains certain sub-level estimates needed to prove Theorem \ref{main1}. Before stating any further results, let us formulate a conjecture:
\begin{conjecture}\label{dimensionConservation} In Proposition \ref{prop1}(a), the hypothesis $\Hd B \leq 1/2$ can be relaxed to $\Hd B \leq 1$. In part (b), the hypothesis $\Hd B \leq 1$ can be relaxed to $\Hd B \leq 2$.
\end{conjecture}

We fall short of proving the conjecture in two ways: first, we are only able to obtain a non-trivial lower bound for the \emph{packing dimension} $\Pd$ (see \cite[\S5.9]{Mat3}) of the projections, and, second, our bound is much weaker than the full dimension conservation conjectured above (the first shortcoming has already been partially overcome in later work, see \cite{OO} and \cite{Or3}). Our first main result is the following:

\begin{thm}\label{main1} Let $B \subset \R^{3}$ be an analytic set, and write $s := \Hd B$.
\begin{itemize}
\item[(a)] If $s > 1/2$, there exists a number $\sigma_{1} = \sigma_{1}(s) > 1/2$ such that $\Pd \rho_{\theta}(B)\geq \sigma_{1}$ almost surely.
\item[(b)] If $s > 1$, there exists a number $\sigma_{2} = \sigma_{2}(s) > 1$ such that $\Pd \pi_{\theta}(B) \geq \sigma_{2}$ almost surely.
\end{itemize}
\end{thm}

\begin{remark}\label{bestBounds} The lower bounds for $\sigma_{1}(s)$ and $\sigma_{2}(s)$ given by the proof of Theorem \ref{main1} are most likely not optimal and certainly not very informative. They are
\begin{displaymath} \sigma_{1}(s) \geq \frac{1}{2} + \frac{1}{2} \cdot \frac{(2s - 1)^{2}}{12s^{2} + 4s - 1}  \quad \text{and} \quad \sigma_{2}(s) \geq 1 + \frac{(s - 1)^{2}}{2s - 1}. \end{displaymath}
For $\sigma_{1}(s)$ we also have the easy bound $\sigma_{1}(s) \geq s/2$, see Proposition \ref{s/2}. The two bounds are equal when $s \approx  1.077$.
\end{remark}

The proof of Theorem \ref{main1} involves analysing the (hypothetical) situation where the packing dimension of the projections of $B$ drops in positively many directions very close to the `classical' bounds given by Proposition \ref{prop1}. Building on this counter assumption, we extract a large subset of $B$ with additional structure. This information is used to show that the projections of the subset must have fairly large dimension.

\subsection{A Marstrand type theorem for self-similar sets}

For self-similar sets in $\R^{3}$ without rotations, we are able to obtain some optimal results, including the part of Conjecture \ref{dimensionConservation} concerning the projections $\rho_{\theta}$. The reason is that such sets $K$ enjoy the following structural property. If $\pi \colon \R^{3} \to V$ is the orthogonal projection onto any plane $V \subset \R^{3}$ and $\varepsilon > 0$, there exists a compact subset $\tilde{K} \subset K$ with $\Hd \tilde{K} \geq \Hd \pi(K) - \varepsilon$ such that the restriction $\pi|_{\tilde{K}}$ is bi-Lipschitz. Our second main result is the following:
\begin{thm}\label{main2} Let $K \subset \R^{3}$ be a self-similar set without rotations.
\begin{itemize}
\item[(a)] If $0 \leq \Hd K \leq 1$, then $\Hd \rho_{\theta}(K) = \Hd K$ almost surely.
\item[(b)] If $\Hd K > 1$, then $\rho_{\theta}(K)$ has positive length almost surely.
\end{itemize}
\end{thm}

The phrase 'self-similar set without rotations' means that the generating similitudes of $K$ have the form $\psi(x) = rx + w$ for some $r \in (0,1)$ and $w \in \R^{3}$. The case of self-similar sets with only 'rational' rotations easily reduces to the one with no rotations, but we are not able to prove Theorem \ref{main2} for arbitrary self-similar sets in $\R^{3}$. Part (a) of Theorem \ref{main2} follows from recent work of M. Hochman \cite{Ho}, but the methods of proof are very different. Part (b) is new. In addition to the structural property of self-similar sets mentioned above, the proof of part (b) relies on an application of Theorem \ref{main1}(b). Unfortunately, the structural property can completely fail for general sets $B \subset \R^{3}$, see Remark \ref{noStructure}, so Theorem \ref{main2} does not seem to admit further generalisation with our techniques.

Specialising Theorem \ref{main2} to the family of lines foliating the surface of a vertical cone in $\R^{3}$, one immediately obtains the following corollary:
\begin{cor}\label{cor1} Let $K \subset \R$ be an equicontractive self-similar set with $\Hd K > 1/3$. Then
\begin{displaymath} (\cos \theta) \cdot K + (\sin \theta) \cdot K + K \end{displaymath}
has positive length for almost all $\theta \in [0,1]$.
\end{cor}
A self-similar set is called \emph{equicontractive}, if all the generating similitudes have equal contraction ratios. The proof can be found in Section \ref{furtherResults}. The corollary is close akin to Theorem 1.1(b), in Y. Peres and B. Solomyak's paper \cite{PS} with the choice $C_{\lambda} = (\cos \lambda) \cdot K + (\sin \lambda) \cdot K$, in the notation of \cite{PS}. The self-similar sets $C_{\lambda}$ treated in in \cite{PS} are not as specific in form as the ones in Corollary \ref{cor1}. On the other hand, the proof in \cite{PS} is based on the concept of \emph{transversality} and, hence, requires that the sets $C_{\lambda}$ satisfy the strong separation condition for all $\lambda \in (0,1)$. This is generally not the case with the sets $C_{\lambda} = (\cos \lambda) \cdot K + (\sin \lambda) \cdot K$ above.

\subsubsection{Notation} Throughout the paper we will write $a \lesssim b$, if $a \leq Cb$ for some constant $C \geq 1$. The two-sided inequality $a \lesssim b \lesssim a$, meaning $a \leq C_{1}b \leq C_{2}a$, is abbreviated to $a \sim b$. Should we wish to emphasise that the implicit constants depend on a parameter $p$, we will write $a \lesssim_{p} b$ and $a \sim_{p} b$. The closed ball in $\R^{d}$ with centre $x$ and radius $r > 0$ will be denoted by $B(x,r)$.

\section{Acknowledgements}

We are grateful to Spyridon Dendrinos, Esa and Maarit J\"arvenp\"a\"a, Pertti Mattila and David Preiss for stimulating discussions and encouragement. We also wish to thank Esa and Maarit J\"arvenp\"a\"a for their hospitality during our visit at the University of Oulu. Finally, we are indebted to an anonymous referee for many comments, which certainly increased the accessibility of the paper.

\section{General sets}

In this section, we prove Proposition \ref{prop1} and Theorem \ref{main1}. It suffices to prove all `almost sure' statements for any fixed compact subinterval $I$ of the parameter set $U$. For the rest of the paper, we assume that this interval is $I = [0,1]$.

\subsection{Classical bounds}\label{s:1/2}

As we mentioned in the introduction, the Hausdorff dimension of an analytic set $B \subset \mathbb{R}^d$ with $\Hd B \leq k$ is preserved under almost every projection onto an $k$-dimensional subspace in $\mathbb{R}^n$. This was proved by J. M. Marstrand \cite{Mar} in 1954 for $d=2$ and $k=1$ and by P. Mattila for general $k < d$ in 1975. In 1968, R.\ Kaufman \cite{Ka} found a "potential theoretic" proof for Marstrand's result, using integral averages over energies of projected measures. It is a natural point of departure for our studies to see if Kaufman's method could be used to prove dimension conservation for non-degenerate families of projections onto lines and planes in $\R^{3}$. For appropriate ranges of $\Hd B$ -- namely when $\Hd B$ is small enough -- the answer is positive. This is the content of Proposition \ref{prop1}.

\begin{proof}[Proof of Proposition \ref{prop1}]

We discuss first part (a) of the proposition, which concerns projections onto lines. Since orthogonal projections are Lipschitz continuous, the upper bound $\Hd \rho_{\theta}(B) \leq \Hd B$ holds for every parameter $\theta$. To establish the almost sure lower bound $\Hd\rho_{\theta}(B) \geq \Hd B$, we use the standard potential-theoretic method. Let $t<\Hd B \leq 1/2$ and find a positive and finite Borel regular measure $\mu$ which is supported on $B$ and whose $t$-energy is finite,
\begin{displaymath}
I_t(\mu):=\int_{B} \int_{B} |x-y|^{-t}d\mu(x)d\mu(y)<\infty.
\end{displaymath}
Such a measure exists by Frostman's lemma for analytic sets, see \cite{Ca}. For each $\theta\in [0,1]$, the push-forward measure $\mu_{\theta} = \rho_{\theta\sharp}\mu$ defined by $\mu_{\theta}(E):=\mu(\rho_{\theta}^{-1}(E))$ is a measure supported on $\rho_{\theta}(B)$. Our goal is to prove that $\int_0^1 I_t(\mu_{\theta}) d\theta<\infty$, which implies $I_t(\mu_{\theta})<\infty$ and thus $\Hd \rho_{\theta}(B) \geq t$
for almost every $\theta\in [0,1]$. Using Fubini's theorem, we find
\begin{align*}
\int_0^{1} I_t(\mu_{\theta})d\theta
&=\int_0^{1} \int_{\rho_{\theta}(B)} \int_{\rho_{\theta}(B)} |u-v|^{-t}d\mu_{\theta}(u)d\mu_{\theta}(v) \, d\theta \\
&= \int_B \int_B \left(\int_0^{1} |\rho_{\theta}(x-y)|^{-t}d\theta \right)d\mu(x)d\mu(y)\\
&\lesssim \int_B \int_B  |x-y|^{-t}d\mu(x)d\mu(y) =I_t(\mu).
\end{align*}
The inequality follows from the next lemma combined with the linearity of $\rho_{\theta}$.

\begin{lemma}\label{l:int_univ_bd}
Given $t<1/2$, the estimate
\begin{displaymath}
\int_0^{1} |\rho_{\theta}(x)|^{-t} \, d\theta\lesssim_{t} 1
\end{displaymath}
holds for all $x\in S^2$.
\end{lemma}

\begin{proof}[Proof of Lemma \ref{l:int_univ_bd}]
We consider the function
\begin{displaymath}
\Pi:[0,1] \times S^2\to \mathbb{R}, \quad \Pi(\theta,x):=\rho_{\theta}(x)= x \cdot \gamma(\theta).
\end{displaymath}
To prove the lemma, we fix $x \in S^{2}$ and study the behaviour of $\theta\mapsto \Pi(\theta,x)$. We note that
\begin{displaymath}
{\partial_{\theta} \Pi} (\theta,x_{0}) = x \cdot \dot{\gamma}(\theta)\quad \text{and}\quad{\partial^2_{\theta} \Pi} (\theta,x) = x \cdot \ddot{\gamma}(\theta).
\end{displaymath}
If $\Pi(\theta,x)={\partial_{\theta} \Pi}(\theta,x)=0$ for some $x\in S^2$ and $\theta\in [0,1]$, we infer that $x$ is orthogonal to both $\gamma(\theta)$ and $\dot{\gamma}(\theta)$. If this happens, the second derivative $\partial^{2}_{\theta} \Pi(\theta,x)$ cannot vanish, because then $x$ would be orthogonal to $\ddot{\gamma}(\theta)$ as well, and this is ruled out by the non-degeneracy condition \eqref{eq:curve_cond}. We have now shown that
\begin{displaymath}
\Pi(\theta,x)={\partial_{\theta} \Pi}(\theta,x)=0\quad \Rightarrow \quad{\partial^2_{\theta} \Pi}(\theta,x)\neq 0
\end{displaymath}
for $(\theta,x)\in [0,1] \times S^2$. A compactness argument then yields a constant $c>0$ such that
\begin{equation}\label{eq:c_funct}
\max\{ |\Pi(\theta,x)|, |\partial_{\theta}\Pi(\theta,x)|,|\partial^2_{\theta}\Pi(\theta,x)|\} \geq c, \quad (\theta,x)\in [0,1]\times S^2.
\end{equation}

Since
\begin{align*}
\int_{0}^1|\Pi(\theta,x)|^{-t} \, d\theta&= \int_0^{\infty}\mathcal{L}^1(\{\theta\in [0,1]:\; |\Pi(\theta,x)|\leq r^{-\frac{1}{t}}\}) \, dr,
\end{align*}
we will need a uniform estimate for the $\mathcal{L}^{1}$ measures of the sub-level sets of $\theta \mapsto |\Pi(\theta,x)|$. Such an estimate is provided for instance by \cite[Lemma 3.3]{Ch}: for every $k\in\mathbb{N}$, there is a constant $C_k<\infty$ so that for every interval $I\subset \R$, every $f\in\mathcal{C}^k(I)$ and every $\lambda>0$,
\begin{equation}\label{eq:sublevel} \mathcal{L}^1(\{\theta\in I:\; |f(\theta)|\leq \lambda\})\leq C_k\left(\frac{\lambda}{\inf_{\theta \in I}|\partial^k_{\theta}f(\theta)|}\right)^{\frac{1}{k}}. \end{equation}
The next lemma,  proved in Appendix \ref{appendix_b}, shows that the mapping $\theta \mapsto \Pi(\theta,x)$ can only have finitely many zeros on $[0,1]$.

\begin{lemma}\label{l:no_zeros}
Let $\gamma:[0,1]\to S^2$ be a parameterised curve satisfying the condition \eqref{eq:curve_cond}. Then there exists $\varepsilon>0$ such that for all $x\in S^2$ the function $\theta \to \rho_{\theta}(x)$ vanishes in at most two points in every interval of length $\varepsilon$.
\end{lemma}

Recall that in each of the finitely many points $\theta_{0} \in [0,1]$, where $\Pi(\theta_{0},x) = 0$, either ${\partial_{\theta} \Pi}(\theta_{0},x)\neq 0$ or ${\partial ^2_{\theta} \Pi}(\theta_{0},x)\neq 0$. Now, the uniform continuity of $\Pi$ and its partial derivatives guarantee that there exists an open ball $U(x)$ centred at $x$ with the following property: $[0,1]$ can be covered by a finite number of intervals $I_1,\ldots,I_{n(x)}$, for each of which there are numbers $k_i \in \{0,1,2\}$ and $c_i>0$ with
\begin{equation}\label{eq:lower_bound}
\inf_{\theta \in I_i} |\partial_{\theta}^{k_i} \Pi(y,\theta)|\geq c_i\quad\text{for all }y\in U(x).
\end{equation}
Set $c_0:= \min\{c_1,\ldots,c_{n(x)}\}$.
The sub-level set estimate \eqref{eq:sublevel} applied to this situation yields
\begin{displaymath}
\mathcal{L}^1(\{\theta\in I_i:\; |\Pi(\theta,y)|\leq \lambda\})\leq \left\{\begin{array}{ll}0 &\text{if }k_i =0,\lambda\leq c_0\\ C_{k_i}\left(\frac{\lambda}{c_i}\right)^{\frac{1}{k_i}}&\text{if }k_i\in \{1,2\}\end{array}\right.
\end{displaymath}
for all $i\in \{1,\ldots,n(x)\}$ and $y\in U(x)$. Hence, there exists a finite constant $c(x)>0$ such that
\begin{displaymath}
\mathcal{L}^1(\{\theta\in [0,1]:\; |\Pi(\theta,y)|\leq \lambda\})\leq c(x)\lambda^{\frac{1}{2}},\quad\text{for all }y\in U(x) \text{ and } 0<\lambda \leq 1.
\end{displaymath}
The sphere $S^2$ can be covered by finitely many balls of the form $U(x)$, so
\begin{equation}\label{eq:sublevel2}
\mathcal{L}^1(\{\theta\in [0,1]:\; |\Pi(\theta,x)|\leq \lambda\})\lesssim \lambda^{\frac{1}{2}},\quad\text{for all }x\in S^2 \text{ and }0<\lambda \leq 1.
\end{equation}
Finally, we obtain
\begin{align*}
\int_{0}^1|\Pi(\theta,x)|^{-t} \, d\theta&= \int_0^{\infty}\mathcal{L}^1(\{\theta\in [0,1]:\; |\Pi(\theta,x)|\leq r^{-\frac{1}{t}}\}) \, dr\lesssim 1 + \int_1^{\infty}r^{-\frac{1}{2 t}} \, dr,
\end{align*}
where the right hand side is finite by the assumption $t < 1/2$.
\end{proof}

It remains to prove part (b) of the proposition. The result follows directly from \cite[Proposition 3.2]{JJLL} or \cite[Theorem 3.2]{JJK}, since the family $(\pi_{\theta})_{\theta \in U}$ is `full' or `non-degenerate' in the senses of \cite{JJLL} and \cite{JJK}. We will need the sub-level estimate \eqref{eq:sub_level_plane} later, so we choose to include the proof.  As in (a), we are reduced to proving the uniform bound
\begin{equation}\label{eq:goal_plan}
\int_0^1|\pi_{\theta}(x)|^{-t} \, d\theta= \int_0^{\infty}\mathcal{L}^1(\{\theta\in [0,1]:\; |\pi_{\theta}(x)|^2\leq r^{-\frac{2}{t}}\}) \, dr\lesssim_{t} 1,\quad x\in S^2,
\end{equation}
valid for $0 < t < 1$. To see this, we note that the function
\begin{displaymath}
F:[0,1]\times S^2 \to \mathbb{R},\quad F(\theta,x):= |\pi_{\theta}(x)|^2 = d(x,\ell_{\theta})^2 = 1-\rho_{\theta}(x)^2
\end{displaymath}
can have at most second order zeros. Indeed, the formulae
\begin{displaymath}
F(\theta,x)=1- (x\cdot \gamma(\theta))^2,\quad \partial_{\theta}F(\theta,x)=- 2(x\cdot \gamma(\theta))(x\cdot \dot{\gamma}(\theta)),\end{displaymath}
\begin{equation}\label{secondDerivative} \partial_{\theta}^2 F(\theta,x)=-2 (x\cdot \dot{\gamma}(\theta))^2-2(x\cdot \gamma(\theta))(x\cdot \ddot{\gamma}(\theta))
\end{equation}
reveal that if $F(\theta,x)=0$ for some $\theta \in [0,1]$ and $x\in S^2$, then $x$ is parallel to $\gamma(\theta)$. This implies that $x\cdot \dot{\gamma}(\theta)=0$ and $x\cdot \ddot{\gamma}(\theta)\neq 0$ since
\begin{displaymath}
\gamma\cdot\gamma=1\quad \Rightarrow \quad\gamma \cdot \dot{\gamma}=0\quad \Rightarrow \quad \gamma \cdot \ddot{\gamma}\neq 0.
\end{displaymath}
This means that $\partial_{\theta}^2 F(\theta,x)\neq 0$, by \eqref{secondDerivative}. Now, tracing the proof of Lemma \ref{l:no_zeros}, we conclude that the number of zeros of the function $\theta \mapsto F(\theta,x)$ is finite for $x\in S^2$. Since
 $|F(\theta,x)|=|\pi_{\theta}(x)|^2$,
 the sub-level estimate \eqref{eq:sublevel} and the compactness of $S^2$ yield
\begin{equation}\label{eq:sub_level_plane} \mathcal{L}^1 (\{\theta \in [0,1]:\; |\pi_{\theta}(x)| \leq \lambda\})\lesssim \lambda, \quad x\in S^2.
\end{equation}
This proves \eqref{eq:goal_plan} for $0 < t < 1$.
\end{proof}

\subsection{Beyond the classical bounds}

It can be read from the proof of Proposition \ref{prop1} above, why the potential theoretic method does not directly extend beyond the dimension ranges $0 \leq \Hd B \leq 1/2$ (for lines) and $0 \leq \Hd B \leq 1$ (for planes). If $x,y \in \R^{3}$ are points such that $(x - y) \perp \gamma(\theta_{0})$ and $(x - y) \perp \dot{\gamma}(\theta_{0})$ for some $\theta_{0} \in (0,1)$, then both the mapping $\theta \mapsto \rho_{\theta}(x - y)$ and its first derivative have a zero at $\theta = \theta_{0}$. This means that
\begin{equation}\label{potentialFailure} \int_{0}^{1} \frac{d\theta}{|\rho_{\theta}(x - y)|^{t}} = \infty \end{equation}
for any $t > 1/2$. Now, if the whole set $B \subset \R^{3}$ is contained on the line perpendicular to $\gamma(\theta_{0})$ and $\dot{\gamma}(\theta_{0})$, then all the differences $x - y$, $x,y \in B$ enjoy the same property. Thus,
\begin{displaymath} \int_{0}^{1} I_{t}(\rho_{\theta\sharp}\mu) \, d\theta = \infty \end{displaymath}
for any $t > 1/2$ and for any Borel measure $\mu$ supported on $B$.

For the projections $\pi_{\theta}$, the situation is not so clear-cut. Again, the direct potential theoretic approach fails, because if $x - y \in \ell_{\theta_{0}} = V_{\theta_{0}}^{\perp}$ for some $\theta_{0} \in (0,1)$, then \eqref{potentialFailure} holds for any $t > 1$, with $\rho_{\theta}$ replaced by $\pi_{\theta}$. But, this time, we do not know if one can construct a set $B \subset \R^{3}$ with $\Hd B > 1$ such that most of the differences $x -y$, $x,y \in B$, lie on the lines $\ell_{\theta}$, $\theta \in [0,1]$. Thus, for all we know, it is still possible that estimates of the form
\begin{equation}\label{hopeful} \int_{0}^{1} I_{s}(\pi_{\theta\sharp}\mu) \, d\theta \lesssim I_{t}(\mu) < \infty \end{equation}
hold for $1 < s < t$ and for suitable chosen measures $\mu$ supported on $B$.

\subsubsection{Proof of Theorem \ref{main1}: a sketch} Being unable to verify an estimate of the form \eqref{hopeful} -- and knowing its impossibility for projections onto lines -- our proof takes a different road. We will now give a heuristic outline of the argument used in the proof of Theorem \ref{main1}(b), before working out the details (the proof of Theorem \ref{main1}(a) is similar but slightly more technical).  We start with the counter assumption that the dimension of the projections $\pi_{\theta}(B)$ drops very close to one in positively many directions $\theta \in [0,1]$. Using this and the non-degeneracy condition, we find two short, disjoint, compact subintervals $I,J \subset [0,1]$ with the following properties:
\begin{itemize}
\item[(i)] The dimension of the projections $\pi_{\theta}(B)$ is very close to one for `almost all' parameters $\theta \in I \cup J$.
\item[(ii)] The surface
\begin{displaymath} C_{I} := \bigcup_{\theta \in I} \ell_{\theta} \end{displaymath}
is `directionally separated' from the lines $\ell_{\theta}$, $\theta \in J$, in the sense that if $x,y \in C_{I}$, then $x - y$ forms a large angle with any such line $\ell_{\theta}$.
\end{itemize}
The next step is to project the set $B$ onto the planes $V_{\theta}$, $\theta \in I$. Because of (i), we know that the projections $\pi_{\theta}$ are, on average, far from bi-Lipschitz. This implies the existence of many differences near the lines $V_{\theta}^{\perp} = \ell_{\theta}$, $\theta \in I$. Building on this information, we find a large subset $\tilde{B} \subset B$ lying entirely in a small neighbourhood of $C_{I}$. The closer the dimension of the projections $\pi_{\theta}(B)$ drops to one for $\theta \in I$, the larger we can choose $\tilde{B}$. Then, we recall (ii) and observe that the differences $x - y$ with $x,y \in \tilde{B}$ are directionally far from the lines $\ell_{\theta}$, $\theta \in J$ (at least if $|x - y|$ is large enough). This means, essentially, that the restrictions $\pi_{\theta}|_{\tilde{B}} \colon \tilde{B} \to \R^{2}$, $\theta \in J$, are bi-Lipschitz and shows that the dimension of $\pi_{\theta}(B)$ exceeds the dimension of $\tilde{B}$ for $\theta \in J$. If the dimension of $\tilde{B}$ was taken close enough to the dimension of $B$, we end up contradicting (i).

\subsection{Proof of Theorem \ref{main1}: the details}\label{ss:details_main1} We will not discuss the proof of Theorem \ref{main1}(a) informally, since the general outline resembles so closely the one in the proof of Theorem \ref{main1}(b). Our first aim is to reduce the proof of Theorem \ref{main1} to verifying a discrete statement, Theorem \ref{mainDiscrete}, which concerns sets and projections at a single scale $\delta > 0$. To this end, we need some definitions.
\begin{definition}[$(\delta,s)$-sets] Let $\delta,s > 0$, and let $P \subset \R^{3}$ be a finite $\delta$-separated set. We say that $P$ is a $(\delta,s)$-set, if it satisfies the estimate
\begin{displaymath} |P \cap B(x,r)| \lesssim \left(\frac{r}{\delta}\right)^{s}, \qquad x \in \R^{3},\: r \geq \delta. \end{displaymath}
Here $| \cdot |$ refers to cardinality, but, in the sequel, it will also be used to denote length in $\R$ and area in $\R^{2}$. This should cause no confusion, since for any set $A$ only one of the possible meanings of $|A|$ makes sense. 
\end{definition}

In a way to be quantified in the next lemma, $(\delta,s)$-sets are well-separated $\delta$-nets inside sets with positive $s$-dimensional Hausdorff content (denoted by $\cH_{\infty}^{s}$). This principle -- a discrete Frostman's lemma -- is most likely folklore, but we could not locate a reference for exactly the formulation we need. So, we choose to include a proof in Appendix \ref{appFrostman}.
\begin{lemma}[Frostman]\label{frostman} Let $\delta,s > 0$, and let $B \subset \R^{3}$ be any set with $\cH^{s}_{\infty}(B) =: \kappa > 0$. Then there exists a $(\delta,s)$-set $P \subset B$ with cardinality $|P| \gtrsim \kappa \cdot \delta^{-s}$.
\end{lemma}

As we have seen, the potential theoretic method cannot be used to improve Proposition \ref{prop1}, because the projections onto planes (resp. lines) may have first (resp. second) order zeros. Such zeros lie on certain `bad lines', the unions of which form `bad cones' in $\R^{3}$. Let us establish notation for these objects.
\begin{definition}[Cones spanned by curves on $S^{2}$] Let $\gamma \colon [0,1] \to S^{2}$ be a curve. If $I \subset [0,1]$ is a compact subinterval, we write
\begin{displaymath} C_{I}(\gamma):= \bigcup_{\theta \in I} \spa(\gamma(\theta)) \subset \R^{3}. \end{displaymath}
\end{definition}

Two special cases of this definition are of particular interest:
\begin{definition}[Bad lines and bad cones for projection families] Let $\gamma \colon U \to S^{2}$ be a non-degenerate curve as in Definition \ref{d:non-deg}, and let $\eta \colon U \to S^{2}$ be the curve
\begin{displaymath} \eta(\theta) := \frac{\gamma(\theta) \times \dot{\gamma}(\theta)}{|\gamma(\theta) \times \dot{\gamma}(\theta)|}. \end{displaymath}
\begin{itemize}
\item[(a)] A \emph{bad line} for the projection family $(\rho_{\theta})_{\theta \in U}$ is any line of the form
\begin{displaymath} b_{\theta} := \spa(\eta(\theta)) \subset \R^{3}, \quad \theta \in U. \end{displaymath}
Unions of bad lines form \emph{bad cones}, as in the previous definition: if $I \subset [0,1]$ is a compact subinterval, we write
\begin{displaymath} C_{I}^{\rho} := C_{I}(\eta). \end{displaymath}
\item[(b)] For the projection family $(\pi_{\theta})_{\theta \in U}$, the \emph{bad lines} have the form
\begin{displaymath} \ell_{\theta} = \spa(\gamma(\theta)). \end{displaymath}
We also define the \emph{bad cones }
\begin{displaymath} C_{I}^{\pi} := C_{I}(\gamma), \qquad I \subset [0,1]. \end{displaymath}
\end{itemize}
\end{definition}

The definitions of bad lines and cones for $(\rho_{\theta})_{\theta \in U}$ and $(\pi_{\theta})_{\theta \in U}$ are closely related with the zeros of the projections. For instance,
\begin{displaymath} x \in b_{\theta_{0}} \quad \Longleftrightarrow \quad x \perp \gamma(\theta_{0}) \text{ and }x \perp \dot{\gamma}(\theta_{0}), \end{displaymath}
where the right hand side is just another way of saying that
\begin{displaymath} \theta \mapsto \rho_{\theta}(x) = \gamma(\theta) \cdot x \quad \text{and} \quad \theta \mapsto \partial_{\theta} \rho_{\theta}(x) = \dot{\gamma}(\theta) \cdot x \end{displaymath}
vanish simultaneously at $\theta = \theta_{0}$. In particular, if $x \in b_{\theta_{0}}$, then
\begin{displaymath} \int_{\theta_{0} - \varepsilon}^{\theta_{0} + \varepsilon} \frac{d\theta}{|\rho_{\theta}(x)|^{t}} = \infty \end{displaymath}
for any $\varepsilon > 0$ and $t > 1/2$. For the projections $\pi_{\theta}$, the situation is even simpler: the mapping $\theta \mapsto |\pi_{\theta}(x)|$ has a (first order) zero at $\theta = \theta_{0}$, if and only if $x \in \ell_{\theta_{0}}$.

Now we can explain how the non-degeneracy hypothesis \eqref{eq:curve_cond} is used in the proof of Theorem \ref{main1}. It will ensure that if $I,J \subset [0,1]$ are appropriately chosen short intervals, then the bad cones $C_{I}^{\rho},C_{J}^{\rho}$ (in part (a)) or $C_{I}^{\pi},C_{J}^{\pi}$ (in part (b)) `point in essentially different directions'. This concept is captured by the next definition:
\begin{definition} Let $\gamma \colon [0,1] \to S^{2}$ be any curve, and let $I,J \subset [0,1]$ be disjoint compact subintervals. We write
\begin{displaymath} C_{I}(\gamma) \not\parallel C_{J}(\gamma), \end{displaymath}
if there is a constant $c = c(\gamma,I,J) > 0$ with the following property. If $x,y \in C_{I}(\gamma)$ and $\xi \in C_{J}(\gamma) \cap S^{2} = \{\gamma(\theta) : \theta \in J\}$, then
\begin{displaymath} \left|\frac{x - y}{|x - y|} - \xi\right| \geq c. \end{displaymath}
An equivalent way to state the condition is to say that there is a constant $L = L(\gamma,I,J) < 1$ such that every orthogonal projection from $C_{I}(\gamma)$ to a line on $C_{J}(\gamma)$ is $L$-Lipschitz. The next lemma shows how to find intervals $I,J \subset [0,1]$ such that $C_{I}(\gamma) \not\parallel C_{J}(\gamma)$.
\end{definition}
\begin{lemma}\label{CICJ} Given a $\mathcal{C}^2$ curve $\gamma:[0,1]\to S^2$ with nowhere vanishing tangent, suppose that $\theta_{1},\theta_{2} \in (0,1)$ are such that $\gamma(\theta_{2}) \notin \spa(\{\gamma(\theta_{1}),\dot{\gamma}(\theta_{1}\})$. Then there exist $\varepsilon_{1},\varepsilon_{2} > 0$ such that
\begin{displaymath} C_{I} \not\parallel C_{J} \end{displaymath}
with $I = [\theta_{1} - \varepsilon_{1}, \theta_{1} + \varepsilon_{1}]$, $J = [\theta_{2} - \varepsilon_{2}, \theta_{2} + \varepsilon_{2}]$, $C_{I} = C_{I}(\gamma)$ and $C_{J} = C_{J}(\gamma)$. \end{lemma}

This result is rather intuitive; a rigorous proof is given in Appendix \ref{appendix_b}. Now we are prepared to formulate a $\delta$-discretised version of Theorem \ref{main1}.

\begin{thm}\label{mainDiscrete}  Let $s > 0$, and let $P \subset B(0,1)$ be a $(\delta,s)$-set with cardinality $|P| \sim \delta^{-s}$. The following statements hold for $\delta > 0$ small enough.
\begin{itemize}
\item[(a)] Suppose that $I,J \subset [0,1]$ are intervals such that
\begin{displaymath} C_{I}^{\rho} \not\parallel C_{J}^{\rho}. \end{displaymath}
If $s > 1/2$, there exist $\varepsilon_{1} = \varepsilon_{1}(s) > 0$ and $\sigma_{1} = \sigma_{1}(s) > 1/2$ with the following property. Suppose that $E_{I} \subset I$ and $E_{J} \subset J$ have lengths $|E_{I}| \geq \delta^{\varepsilon_{1}}$ and $|E_{J}| \geq \delta^{\varepsilon_{1}}$. Then there exists a direction $\theta \in E_{I} \cup E_{J}$ such that
\begin{displaymath} |\rho_{\theta}(P(\delta))| \geq \delta^{1 - \sigma_{1}}. \end{displaymath}
\item[(b)] Suppose that $I,J \subset [0,1]$ are intervals such that
\begin{displaymath} C_{I}^{\pi} \not\parallel C_{J}^{\pi}. \end{displaymath}
If $s > 1$, there exist $\varepsilon_{2} = \varepsilon_{2}(s) > 0$ and $\sigma_{2} = \sigma_{2}(s) > 1$ with the following property. Suppose that $E_{I} \subset I$ and $E_{J} \subset J$ have lengths $|E_{I}| \geq \delta^{\varepsilon_{2}}$ and $|E_{J}| \geq \delta^{\varepsilon_{2}}$. Then there exists a direction $\theta \in E_{I} \cup E_{J}$ such that
\begin{displaymath} |\pi_{\theta}(P(\delta))| \geq \delta^{2 - \sigma_{2}}. \end{displaymath}
\end{itemize}
\end{thm}

Let us briefly explain how Theorem \ref{main1} follows from its $\delta$-discretised variant. First, we note that in order to derive statements like Theorem \ref{main1} for the packing dimension of projections, it suffices to prove their analogues for the \emph{upper box dimension} $\Bd$, defined by
\begin{displaymath} \Bd R = \limsup_{\delta \to 0} \frac{\log N(R,\delta)}{-\log \delta} \end{displaymath}
for bounded sets $R \subset \R^{d}$, where $N(R,\delta)$ is the least number of balls of radius $\delta$ required to cover $R$. This reduction is possible thanks to the following lemma.
\begin{lemma}\label{packingToBox} Let $\sigma > 0$, let $\mu$ be a Borel regular measure, and let $B \subset \R^{3}$ be a $\mu$-measurable set such that $\mu(B) > 0$, and
\begin{displaymath} |\{\theta \in [0,1] : \Pd \rho_{\theta}(B) < \sigma\}| > 0. \end{displaymath}
Then there exists a compact set $K \subset B$ with $\mu(K) > 0$ such that
\begin{displaymath} |\{\theta \in [0,1] : \Bd \rho_{\theta}(K) < \sigma\}| > 0. \end{displaymath}
\end{lemma}

\begin{proof} The proof is the same as that of \cite[Lemma 4.5]{Or2}, except for some obvious changes in notation. \end{proof}

The statement also holds with the projections $\rho_{\theta}$ replaced by $\pi_{\theta}$. Now, if Theorem \ref{main1} failed for $\Pd$, there would exist an analytic set $B \subset \R^{3}$ with $\Hd B > s$ such that the projections of $B$ have packing dimension less than $\sigma \in \{\sigma_{1},\sigma_{2}\}$ in a set of directions of positive measure. Then, we could find a Frostman measure $\mu$ inside $B$ and apply Lemma \ref{packingToBox} to $B$, $\mu$ and $\sigma$. The conclusion would be that also the $\Bd$-variant of Theorem \ref{main1} has to fail in a set of directions of positive measure.

\begin{proof}[Proof of Theorem \ref{main1}] We will now describe how to use Theorem \ref{mainDiscrete}(a) to prove the $\Bd$-variant of Theorem \ref{main1}(a). The deduction of Theorem \ref{main1}(b) from Theorem \ref{mainDiscrete}(b) is analogous. To reach a contradiction, suppose that $\Hd B = s > 1/2$, but there is a positive length subset $E \subset [0,1]$ such that
\begin{equation}\label{form2} \Bd \rho_{\theta}(B) < \sigma_{1} - c \end{equation}
for all $\theta \in E$ and some small constant $c > 0$. Here $\sigma_{1} = \sigma_{1}(s) > 1/2$ is the constant from Theorem \ref{mainDiscrete}. Fix a Lebesgue point $\theta_{1} \in E$, and then choose another Lebesgue point $\theta_{2} \in E$ such that
\begin{equation}\label{form1} \eta(\theta_{2}) \notin \spa(\{\eta(\theta_{1}), \dot{\eta}(\theta_{1})\}), \end{equation}
where $\eta = \gamma \times \dot{\gamma}/|\gamma \times \dot{\gamma}|$. This is precisely where we need the non-degeneracy hypothesis. \begin{lemma}\label{etaNonDeg}
Let $\gamma: U\to S^2$ be a $\mathcal{C}^3$ curve satisfying the non-degeneracy condition \eqref{eq:curve_cond}. Then the curve $\eta: U\to S^2$, given by $\eta:= \frac{\gamma \times \dot{\gamma}}{|\gamma \times \dot{\gamma}|}$, fulfills the same condition, that is,
\begin{equation}\label{eq:non_deg_lambda}
\mathrm{span}\{\eta(\theta),\dot{\eta}(\theta),\ddot{\eta}(\theta)\}=\mathbb{R}^3,
\end{equation}
for every $\theta \in U$.
\end{lemma}
It follows from this lemma, which is proved in Appendix \ref{appendix_b}, and from Lemma \ref{l:no_zeros}
that for any given $2$-plane $W \subset \R^{2}$ there are only finitely many choices of $\theta_{2} \in [0,1]$ such that $\eta(\theta_{2}) \in W$: indeed, if $\bar{n}$ is the normal vector of the plane $W$, Lemma \ref{l:no_zeros} implies that the mapping $\theta \mapsto \eta(\theta) \cdot \bar{n}$ can only have a bounded number of zeros $\theta \in [0,1]$. Now we may apply Lemma \ref{CICJ} to the path $\eta$: thus, we find disjoint compact intervals $I \ni \theta_{1}$ and $J \ni \theta_{2}$ with the property that
\begin{displaymath} C_{I}^{\rho} \not\parallel C_{J}^{\rho}. \end{displaymath}
This places us in a situation, where we can apply Theorem \ref{mainDiscrete}(a). Let $\varepsilon_{1} > 0$ be the number defined there, and let $\delta > 0$ be so small the lengths of  $E_{I} := E \cap I$ and $E_{J} := E \cap J$ exceed $\delta^{\varepsilon_{1}}$. Then use Lemma \ref{frostman} to find a $(\delta,s)$-set $P \subset B$ with cardinality $|P| \sim \delta^{-s}$. From \eqref{form2}, we see that
\begin{displaymath} |\rho_{\theta}(P(\delta))| \leq |\rho_{\theta}(B(\delta))| \lesssim \delta^{1-\sigma_{1} + c} \end{displaymath}
for $\delta > 0$ and $\theta \in E_{I} \cup E_{J}$. For $\delta > 0$ small enough, this is incompatible with the conclusion of Theorem \ref{mainDiscrete}(a).
\end{proof}

It remains to prove Theorem \ref{mainDiscrete}. The basic approach for both (a) and (b) is the same, but (b) is slightly simpler from a technical point of view. This is why we choose to give the proof of (b) first.

\begin{proof}[Proof of Theorem \ref{mainDiscrete}(b)] Recall that $P \subset B(0,1)$ is a $(\delta,s)$-set of cardinality $|P| \sim \delta^{-s}$. We make the counter assumption that
\begin{equation}\label{form3} |\pi_{\theta}(P(\delta))| < \delta^{2 - \sigma_{2}} \end{equation}
for all $\theta \in E_{I} \cup E_{J}$. The constant $\sigma_{2} \in (1,s)$ will be fixed as the proof progresses. In particular, \eqref{form3} means that for $\theta \in E_{J}$, the projection $\pi_{\theta}(P)$ can be covered by $\lesssim \delta^{-\sigma_{2}}$ discs of radius $\delta > 0$.

For $x,y \in \R^{3}$ and $\theta \in E_{I} \cup E_{J}$, we define the relation $x \sim_{\theta} y$ as follows:
\begin{displaymath} x \sim_{\theta} y \quad \Longleftrightarrow \quad x \neq y\: \text{ and }\: |\pi_{\theta}(x) - \pi_{\theta}(y)| \leq \delta, \end{displaymath}
We also write
\begin{displaymath} T_{I}(x,y) := |\{\theta \in E_{I} : x \sim_{\theta} y\}|. \end{displaymath}
Our first aim is to use \eqref{form3} to find a lower bound for the quantity
\begin{displaymath} \cE := \sum_{x,y \in P} T_{I}(x,y) = \int_{E_{I}} |\{(x,y) \in P^{2} : x \sim_{\theta} y\}|\, d\theta \end{displaymath}
Fix $\theta \in E_{I}$ and choose a minimal (in terms of cardinality) collection of disjoint discs $D_{1},\ldots,D_{M(\theta)} \subset \R^{2}$ such that $\diam(D_{j}) = \delta$ and
\begin{equation}\label{discProperty} \left|P \cap \bigcup_{j = 1}^{M(\theta)} \pi_{\theta}^{-1}(D_{j}) \right| \gtrsim |P| \sim \delta^{-s}. \end{equation}
Then \eqref{form3} implies that $M(\theta) \lesssim \delta^{-\sigma_{2}}$. Next, we discard all the discs $D_{j}$ such that $|P \cap \pi_{\theta}^{-1}(D_{j})| \leq 1$. This way only $\lesssim \delta^{-\sigma_{2}}$ points are deleted from the left hand side of \eqref{discProperty}, so the inequality remains valid for the remaining collection of discs, and for small $\delta > 0$. The point of the discarding process is simply to ensure that
\begin{displaymath} |\{(x,y) \in [P \cap \pi_{\theta}^{-1}(D_{j})]^{2} : x \sim_{\theta} y\}| = |P \cap \pi_{\theta}^{-1}(D_{j})|^{2} - |P \cap \pi_{\theta}^{-1}(D_{j})| \gtrsim |P \cap \pi_{\theta}^{-1}(D_{j})|^{2} \end{displaymath}
for the remaining discs $D_{j}$. This in mind, we estimate $\cE$ from below:
\begin{align} \cE & \gtrsim \int_{E_{I}} \sum_{\text{remaining } D_{j}} |P \cap \pi_{\theta}^{-1}(D_{j})|^{2} \, d\theta \notag \\
& \geq \int_{E_{I}} \frac{1}{M(\theta)} \left(\sum_{\text{remaining } D_{j}} |P \cap \pi_{\theta}^{-1}(D_{j})|\right)^{2} \, d\theta \notag\\
&\label{sizeE} \gtrsim |E_{I}| \cdot \delta^{\sigma_{2}} \cdot |P|^{2} \gtrsim \delta^{\sigma_{2} + \varepsilon_{2} - 2s}. \end{align}
Our second aim is to show that \eqref{sizeE} gives some structural information about $P$, if $\varepsilon_{2}$ and $\sigma_{2}$ are small. For $x \in P$ we define a `neighbourhood' $N(x)$ of $x$ by
\begin{displaymath} N(x) := P \cap (x + C_{I}^{\pi}(2\delta)). \end{displaymath}
Recall that $C_{I}^{\pi}$ was defined as the union of the lines $\ell_{\theta}$, $\theta \in I$, perpendicular to the planes $V_{\theta}$. The reason for defining $N(x)$ as we do is the following: if $y \in P \setminus N(x)$, then $y - x \notin C_{I}^{\pi}(2\delta)$, so that the difference $y - x$ stays at distance $> \delta$ from any of the orthogonal complements of the planes $V_{\theta}$. In particular, it is not possible that $x \sim_{\theta} y$ for any parameter $\theta \in I$, which implies that
\begin{equation}\label{form4} \cE = \sum_{x \in P} \sum_{y \in N(x)} T_{I}(x,y). \end{equation}
To connect the sizes of the neighbourhoods $N(x)$ with \eqref{sizeE}, we need a universal estimate for $T_{I}(x,y)$:
\begin{lemma}\label{universalB} Let $x,y \in \R^{3}$ be $\delta$-separated points. Then
\begin{displaymath} T_{I}(x,y) \lesssim \frac{\delta}{|x - y|}. \end{displaymath}
\end{lemma}
\begin{proof} Apply the sub-level estimate \eqref{eq:sub_level_plane} with $\lambda = \delta$.
\end{proof}

Now we are equipped to search for a large set $N(x) \subset P$. Suppose that $|N(x)| \leq \delta^{-s + \varepsilon}$ for every $x \in P$, where
\begin{displaymath} \varepsilon = \frac{s(s - 1)}{2s - 1}. \end{displaymath}
 Write $A_{j}(x) := \{y \in \R^{3} : 2^{j} \leq |y - x| \leq 2^{j + 1}\}$. Recalling \eqref{form4} and using the inequality
\begin{displaymath} \min\{a,b\} \leq a^{1 - \tfrac{1}{s}}b^{\tfrac{1}{s}}, \quad a,b \geq 0, \end{displaymath}
we estimate as follows:
\begin{align*} \cE & = \sum_{x \in P} \sum_{\delta \leq 2^{j} \leq 1} \sum_{y \in A_{j}(x) \cap N(x)} T_{I}(x,y)\\
& \lesssim \delta \sum_{x \in P} \sum_{\delta \leq 2^{j} \leq 1} 2^{-j} \min\{|N(x)|, |P \cap B(x,2^{j + 1})|\}\\
& \lesssim \delta \sum_{x \in P} \sum_{\delta \leq 2^{j} \leq 1} 2^{-j} \min\left\{\delta^{-s + \varepsilon}, \left(\frac{2^{j}}{\delta}\right)^{s}\right\}\\
& \leq \sum_{x \in P} \sum_{\delta \leq 2^{j} \leq 1} \delta^{(-s + \varepsilon)(1 - \tfrac{1}{s})} \sim \delta^{-2s + \varepsilon(1 - \tfrac{1}{s}) + 1} \cdot \log\left(\frac{1}{\delta}\right). \end{align*}
So, assuming that $|N(x)| \leq \delta^{-s + \varepsilon}$, we can combine the bound above with \eqref{sizeE} to conclude that
\begin{displaymath} \delta^{\sigma_{2} + \varepsilon_{2} - 2s} \lesssim \delta^{-2s + \varepsilon(1 - \tfrac{1}{s}) + 1} \cdot \log\left(\frac{1}{\delta}\right). \end{displaymath}
Since the implicit constants are independent of $\delta > 0$, this shows that either
\begin{itemize}
\item[(i)] There exists a point $x \in P$ with $|N(x)| \geq \delta^{-s + \varepsilon}$, or
\item[(ii)] $\sigma_{2} + \varepsilon_{2} \geq 1 + \varepsilon(1 - \tfrac{1}{s})$.
\end{itemize}
The proof of Theorem \ref{mainDiscrete}(b) nears its end. Our next lemma will show that the projections $\pi_{\theta}|_{N(x)}$, $\theta \in J$, are essentially bi-Lipschitz, so the counter assumption $|\pi_{\theta}(P(\delta))| < \delta^{2 - \sigma_{2}}$ for $\theta \in J$ will force the inequality $|N(x)| \lesssim \delta^{-\sigma_{2}}$. In case (i) holds, this shows that
\begin{displaymath} \sigma_{2} \geq s - \varepsilon = 1 + \frac{(s - 1)^{2}}{2s - 1}. \end{displaymath}
If (i) fails, we conclude from (ii) that
\begin{equation}\label{form5} \sigma_{2}  + \varepsilon_{2} \geq 1 + \varepsilon(1 - \tfrac{1}{s}) = 1 + \frac{(s - 1)^{2}}{2s - 1}. \end{equation}
Either way, Theorem \ref{mainDiscrete}(b) is true for any pair $(\sigma_{2},\varepsilon_{2})$ satisfying \eqref{form5}.

It remains to state and prove the bi-Lipschitz lemma. In order to make the same lemma useful in the proof of Theorem \ref{mainDiscrete}(a), we state a slightly more general version than we would need here.
\begin{lemma}\label{geometry} Assume that $\gamma \colon [0,1] \to S^{2}$ is a curve, and $C_{I} \not\parallel C_{J}$ for some intervals $I,J \subset [0,1]$, where $C_{I} = C_{I}(\gamma)$ and $C_{J} = C_{J}(\gamma)$. Let $x \in \R^{3}$, $\tau > 0$. Then, there exists a constant $C \geq 1$, depending only on $\gamma$, $I$ and $J$, such that whenever $y,y' \in B(0,1)$ satisfy
\begin{displaymath} y,y' \in x + C_{I}(\delta^{\tau}) \quad \text{and} \quad |y - y'| \geq C\delta^{\tau}, \end{displaymath}
then
\begin{equation}\label{angSeparation} \left|\frac{y - y'}{|y - y'|} - \xi\right| \gtrsim_{\gamma,I,J} 1, \qquad \xi \in C_{J} \cap S^{2}. \end{equation}
 \end{lemma}
 \begin{proof} Let $c > 0$ be the constant from the definition of $C_{I}(\gamma) \not\parallel C_{J}(\gamma)$: thus, if $u,v \in x + C_{I}$, then
 \begin{displaymath} \left|\frac{u - v}{|u - v|} - \xi\right| \geq c \end{displaymath}
 for any vector $\xi \in C_{J} \cap S^{2}$. Suppose that $y,y' \in B(0,1)$ satisfy the hypotheses of the lemma, and find $y_{0},y_{0}' \in x + C_{I}$ such that $|y - y_{0}| \leq \delta^{\tau}$ and $|y' - y_{0}'| \leq \delta^{\tau}$. Note that the points $y_{0}$ and $y_{0}'$ are at least $C\delta^{\tau}/2$ apart for $C \geq 10$, and
  \begin{equation}\label{form14} \left|\frac{y - y'}{|y - y'|} - \xi\right| \geq c - \left|\frac{y_{0} - y'_{0}}{|y_{0} - y'_{0}|} - \frac{y - y'}{|y - y'|}\right|, \qquad \xi \in C_{J} \cap S^{2}. \end{equation}
 To estimate the negative term, consider the mapping $b \colon \R^{3} \setminus B(0,C\delta^{\tau}/2) \to S^{2}$, defined by $b(x) = x/|x|$. Choosing $C$ large enough, the mapping $b$ can be made $L$-Lipschitz with $L \leq c\delta^{-\tau}/4$, so
 \begin{align*} \left|\frac{y_{0} - y'_{0}}{|y_{0} - y'_{0}|} - \frac{y - y'}{|y - y'|}\right| & = |b(y_{0} - y'_{0}) - b(y - y')|\\
 & \leq \frac{c\delta^{-\tau}}{4}(|y - y_{0}| + |y' - y_{0}'|) \leq  \frac{c}{2}. \end{align*}
 This and \eqref{form14} give \eqref{angSeparation}.
 \end{proof}

 In the proof of Theorem \ref{mainDiscrete}(b), we apply the lemma with the non-parallel bad cones $C_{I}^{\pi}$ and $C_{J}^{\pi}$ and with $\tau = 1$: let $C \gtrsim_{\gamma,I,J}$ be the constant appearing in the statement of the lemma. If option (i) above is realised, we choose a $C\delta$-net $\tilde{P} \subset N(x)$. Then $|\tilde{P}| \gtrsim \delta^{-s + \varepsilon}$, and the angle between any difference $y - y'$, $y,y' \in \tilde{P}$, and any line $\ell_{\theta} = V_{\theta}^{\perp}$, $\theta \in J$, is bounded from below by a constant. This means that the restrictions $\pi_{\theta}|\tilde{P}$ are bi-Lipschitz, so $|\pi_{\theta}(\tilde{P}(\delta))| \gtrsim \delta^{2 - s + \varepsilon}$ for any $\theta \in J$. The proof of Theorem \ref{mainDiscrete}(b) is now completed in the manner we described above.
\end{proof}

Next, we turn to the proof of Theorem \ref{mainDiscrete}(a). The structure will be familiar, but there are some additional steps to take.

\begin{proof}[Proof of Theorem \ref{mainDiscrete}(a)] All the way down to the lower energy estimate \eqref{sizeE} the argument follows the proof of Theorem \ref{main1}(b) with the obvious changes
\begin{displaymath} \pi_{\theta} \rightsquigarrow \rho_{\theta}, \quad \varepsilon_{2} \rightsquigarrow \varepsilon_{1}, \quad \sigma_{2} \rightsquigarrow \sigma_{1}, \end{displaymath}
and choosing the sets $D_{1},\ldots,D_{M(\theta)}$ as $\delta$-intervals in $\R$ rather than $\delta$-discs in $\R^{2}$. The analogue of \eqref{form4} is
\begin{equation}\label{form6} \cE \gtrsim \delta^{\sigma_{1} + \varepsilon_{1} - 2s}. \end{equation}
The first essential difference appears in the definition of the `neighbourhoods' $N(x)$, $x \in P$. This time
\begin{displaymath} N(x) := P \cap (x + C_{I}^{\rho}(\delta^{\tau})), \end{displaymath}
where $\tau \in (0,1/2)$ is a parameter to be chosen soon. Recall that $C_{I}^{\rho}$ is the union of the bad lines $b_{\theta}$, $\theta \in I$, spanned by the vectors $\gamma(\theta) \times \dot{\gamma}(\theta)$. Contrary to what we did in part (b), if a point $y \in P$ stays away from a neighbourhood $N(x)$, we may \textbf{not} conclude that $x \not\sim_{\theta} y$ for all $\theta \in I$. Instead, the event $y \notin N(x)$ signifies that the mapping $\theta \mapsto \rho_{\theta}(x - y)$ does not have a second order zero on the interval $I$. Consequently, we have an improved estimate for $T_{I}(x,y)$. An `improved estimate' means an improvement over the following universal bound, analogous to the one in Lemma \ref{universalB}:
\begin{lemma}\label{universalA} Let $x,y \in \R^{3}$ be $\delta$-separated points. Then
\begin{displaymath} T_{I}(x,y) \lesssim \left(\frac{\delta}{|x - y|}\right)^{1/2}. \end{displaymath}
\end{lemma}

\begin{proof} Apply the sub-level estimate \eqref{eq:sublevel2} with $\lambda = \delta$. \end{proof}

\begin{lemma}\label{improvement} Suppose that $0 \leq \tau < 1$, and $x,y \in \R^{3}$ satisfy
\begin{displaymath} y - x \notin C_{I}^{\rho}(\delta^{\tau}). \end{displaymath}
Then
\begin{displaymath} T_{I}(x,y) \lesssim \delta^{1 - \tau}. \end{displaymath}
\end{lemma}
\begin{proof} The condition $y - x \notin C_{I}^{\rho}(\delta^{\tau})$ is another way of saying that $d(y - x, b_{\theta}) \geq \delta^{\tau}$ for all bad lines $b_{\theta} = \spa\{\gamma(\theta) \times \dot{\gamma}(\theta)\} \subset C_{I}^{\rho}$. Now, note that the distance of a vector $z$ from $b_{\theta}$ equals the length of the projection $\tilde{\pi}_{\theta}(z)$ of $z$ onto the plane $b_{\theta}^{\perp} = \spa(\{\gamma(\theta),\dot{\gamma}(\theta)\})$. Hence
\begin{displaymath} \left[((x - y) \cdot \gamma(\theta))^{2} + \left((x - y) \cdot \frac{\dot{\gamma}(\theta)}{|\dot{\gamma}(\theta)|}\right)^{2}\right]^{1/2} = |\tilde{\pi}_{\theta}(x - y)| \geq \delta^{\tau}. \end{displaymath}
Since $|\dot{\gamma}(\theta)|$ is bounded from below on $I$, and $\tau < 1$, we may infer that
\begin{displaymath} |\rho_{\theta}(x - y)| \leq 2\delta \quad \Longrightarrow \quad |\partial_{\theta} \rho_{\theta}(x - y)| \gtrsim \delta^{\tau}. \end{displaymath}
This implies that the set $\{\theta \in I : |\rho_{\theta}(x - y)| <  2\delta\}$ consists of intervals $I_{1},\ldots,I_{N}$ around the zeros of $\theta \mapsto \rho_{\theta}(x - y)$ on $I$, and possibly two intervals having a common endpoint with $I$. We saw in Lemma \ref{l:no_zeros} that the number of zeros of $\theta \mapsto \rho_{\theta}(x - y)$, $x \neq y$, on any compact subinterval of $U$ is bounded by a constant independent of $x - y$. So, in order to estimate the length of $\{\theta \in I : |\rho_{\theta}(x - y)| < 2\delta\}$ -- and the cardinality of $T_{I}(x,y)$ -- it suffices to bound the lengths of the intervals $I_{i}$. But the lower bound of the derivative $\partial_{\theta} \rho_{\theta}(x - y)$ readily shows that
\begin{displaymath} |I_{i}| \lesssim \delta^{1 - \tau}, \end{displaymath}
which completes the proof of the lemma. \end{proof}

Next, as in the proof of Theorem \ref{mainDiscrete}(b), we claim that the lower bound \eqref{form6} forces a dichotomy: either $\varepsilon_{1}$ and $\sigma_{1}$ are large, or there exists a neighbourhood $N(x)$ with cardinality $|N(x)| \geq \delta^{-s + \varepsilon_{I}}$. Here
\begin{displaymath} \varepsilon_{I} = \frac{s(2s - 1)}{12s^{2} + 4s - 1} - \kappa, \end{displaymath}
where $\kappa > 0$ is arbitrary (but so small that $\varepsilon_{I} > 0$). Let us first estimate $\cE$ from above, \textbf{assuming} $|N(x)| \leq \delta^{-s + \varepsilon_{I}}$ for every $x \in P$:
\begin{displaymath} \cE = \sum_{x \in P} \sum_{y \in N(x)} T_{I}(x,y) + \sum_{x \in P} \sum_{y \in P \setminus N(x)} T_{I}(x,y) =: S_{1} + S_{2}. \end{displaymath}
The sum $S_{1}$ is bounded using the universal bound in Lemma \ref{universalA}, combined with the size estimate for $|N(x)|$:
\begin{align*} S_{1} & \lesssim \delta^{1/2} \sum_{x \in P} \sum_{\delta \leq 2^{j} \leq 1} 2^{-j/2} \min\{|N(x)|, |P \cap B(x,2^{j + 1})|\}\\
& \lesssim \delta^{1/2} \sum_{x \in P} \sum_{\delta \leq 2^{j} \leq 1} 2^{-j/2} \min\left\{\delta^{-s + \varepsilon_{I}}, \left(\frac{2^{j}}{\delta}\right)^{s}\right\}\\
& \leq \sum_{x \in P} \sum_{\delta \leq 2^{j} \leq 1} \delta^{(-s + \varepsilon_{I})(1 - \tfrac{1}{2s})} \sim \delta^{-2s + \varepsilon_{I}(1 - \tfrac{1}{2s}) + 1/2} \cdot \log\left(\frac{1}{\delta}\right). \end{align*}
To estimate $S_{2}$, we set
\begin{displaymath} \tau = 1/2 - (\varepsilon_{I} + \varepsilon_{J})(1 - \tfrac{1}{2s}) > 0, \end{displaymath}
where
\begin{displaymath} \varepsilon_{J} := \frac{4s^{2}}{12s^{2} + 4s - 1} > 0, \end{displaymath}
and apply Lemma \ref{improvement} with this particular choice of $\tau$:
\begin{displaymath} S_{2} \leq \sum_{x \in P} \sum_{y \in P \setminus N(x)} \delta^{1 - \tau} \lesssim \delta^{-2s + (\varepsilon_{I} + \varepsilon_{J})(1 - \tfrac{1}{2s}) + 1/2}. \end{displaymath}
With our choices of parameters, we see that $\cE = S_{1} + S_{2} \lesssim \delta^{-2s + \varepsilon_{I}(1 - \tfrac{1}{2s}) + 1/2} \cdot \log\left(\frac{1}{\delta}\right)$. Comparing this upper bound with \eqref{form6}, we conclude that one of the following options must hold:
\begin{itemize}
\item[(i)] There exists $x \in P$ with $|N(x)| \geq \delta^{-s + \varepsilon_{I}}$, or
\item[(ii)] $\sigma_{1} + \varepsilon_{1} \geq 1/2 + \varepsilon_{I}(1 - \tfrac{1}{2s})$.
\end{itemize}
Indeed, we obtained (ii) by assuming on the previous page that (i) fails for every $x \in P$. Now, (ii) would directly lead to a lower bound for $\sigma_{1}$ (we will work out the numbers soon), so (i) is the "hard" case. Thus, we momentarily assume that (i) holds and see where we end up. It is time to start using the information about the size of the projections $\rho_{\theta}(P(\delta))$, $\theta \in J$. Write $\tilde{P} := N(x)$, where $|N(x)| \geq \delta^{-s + \varepsilon_{I}}$. Since $\tilde{P} \subset P$, we know that $|\rho_{\theta}(\tilde{P}(\delta))| \leq \delta^{1 - \sigma_{2}}$ for $\theta \in E_{J}$. Consequently, if we set
\begin{displaymath} T_{J} := |\{\theta \in E_{J} : x \sim_{\theta} y\}| \end{displaymath}
and define
\begin{displaymath} \cE_{J} := \sum_{x,y \in \tilde{P}} T_{J}(x,y), \end{displaymath}
the same argument that gave \eqref{form6} yields the lower bound
\begin{equation}\label{form7} \cE_{J} \gtrsim \delta^{\sigma_{2} + \varepsilon_{2} + 2\varepsilon_{I} - 2s}. \end{equation}
With this in mind, we set hunting for a large neighbourhood
\begin{displaymath} N_{J}(y) := \tilde{P} \cap (y + C_{J}^{\rho}(\delta^{\tau})), \quad y \in \tilde{P}. \end{displaymath}
The parameter $\tau > 0$ is the same as before. If all such neighbourhoods have size $|N_{J}(y)| \leq \delta^{-s + \varepsilon_{I} + \varepsilon_{J}}$, precisely the same argument as above shows that
\begin{equation}\label{form8} \cE_{J} \lesssim \delta^{-2s + (\varepsilon_{I} + \varepsilon_{J})(1 - \tfrac{1}{2s}) + 1/2} \cdot \log\left(\frac{1}{\delta}\right). \end{equation}
Indeed, one only needs to observe that the bounds for $T_{I}(x,y)$ in Lemmas \ref{universalA} and \ref{improvement} transfer without change to bounds for $T_{J}(x,y)$. Comparing \eqref{form7} and \eqref{form8}, we arrive at a familiar alternative:
\begin{itemize}
\item[(i')] There exists $y \in \tilde{P}$ with $|N_{J}(y)| \geq \delta^{-s + \varepsilon_{I} + \varepsilon_{J}}$, or
\item[(ii')] $\sigma_{1} + \varepsilon_{1} \geq 1/2 + (\varepsilon_{I} + \varepsilon_{J})(1 - \tfrac{1}{2s}) - 2\varepsilon_{I}$.
\end{itemize}
Now the proof is nearly complete. The next step is to show that (i) and (i') are mutually incompatible with our choices of $\varepsilon_{I}$ and $\varepsilon_{J}$. Consequently, from our alternatives, we will see that either (ii) holds, or (i) \textbf{and} (ii') hold. Both options will lead to a lower bound for $\sigma_{1}$.

To establish the incompatibility of (i) and (i'), we apply Lemma \ref{geometry} with $\tau$ and the non-parallel cones $C_{I}^{\rho}, C_{J}^{\rho}$. Let $C \gtrsim_{\gamma,I,J} 1$ be a constant, which appears in the lemma with these parameters. Assuming (i) and (i'), recalling the formulae for $\varepsilon_{I}$ and $\varepsilon_{J}$, and using the fact that $N_{J}(y)$ is a $(\delta,s)$-set, we have
\begin{align*} |N_{J}(y) \cap B(y,C\delta^{\tau})| & \lesssim_{C} \left(\frac{\delta^{\tau}}{\delta}\right)^{s} = \delta^{(\tau - 1)s}\\
& = \delta^{-s/2 - (\varepsilon_{I}  + \varepsilon_{J})(s - 1/2)}\\
& = \delta^{\kappa(s + 1/2)} \cdot \delta^{-s + \varepsilon_{I} + \varepsilon_{J}}\\
& \leq \delta^{\kappa(s + 1/2)} \cdot |N_{J}(y)|. \end{align*}
This shows that no matter how large $C$ is, for small enough $\delta > 0$ the set $N_{J}(y)$ cannot be contained in the ball $B(y,C\delta^{\tau})$. So, if (i) and (i') hold, and $\delta > 0$ is small enough, we can find a point
\begin{displaymath} y' \in N_{J}(y) \subset \tilde{P} \subset x + C_{I}^{\rho}(\delta^{\tau}) \end{displaymath}
with
\begin{equation}\label{form9} |y - y'| > C\delta^{\tau}. \end{equation}
We infer from Lemma \ref{geometry} that
\begin{equation}\label{form10} \left|\frac{y - y'}{|y - y'|} - \xi\right| \gtrsim_{\gamma,I,J} 1, \qquad \xi \in C_{J}^{\rho} \cap S^{2}. \end{equation}
On the other hand, we know that $y' \in N_{J}(y) \subset y + C_{J}^{\rho}(\delta^{\tau})$, so there is a line $b_{\theta} \subset C_{J}^{\rho}$, $\theta \in J$, such that $d(y - y',b_{\theta}) \leq \delta^{\tau}$. If $b_{\theta}$ is spanned by the unit vector $\xi \in C_{J}^{\rho} \cap S^{2}$, one can combine \eqref{form9} with elementary geometry to show that
\begin{displaymath} \left|\frac{y - y'}{|y - y'|} - \xi\right| \lesssim \frac{1}{C}, \end{displaymath}
as long as $C \leq \delta^{-\tau}$. This is incompatible with \eqref{form10}, if $C$ is large enough (still depending only on $\gamma$, $I$ and $J$). We have established that (i) and (i') cannot hold simultaneously. Thus, if (i) holds, we may infer that also (ii') holds, so $\sigma_{1} + \varepsilon_{1}$ must satisfy the lower bound
\begin{displaymath} \sigma_{1} + \varepsilon_{1} \geq \frac{1}{2} + (\varepsilon_{I} + \varepsilon_{J})\left(1 - \frac{1}{2s}\right) - 2\varepsilon_{I} = \frac{1}{2} + \frac{1}{2} \cdot \frac{(2s - 1)^{2}}{12s^{2} + 4s - 1} + \kappa\left(1 + \frac{1}{2s}\right) \end{displaymath} But if (i) fails, we know that (ii) holds, and then
\begin{displaymath} \sigma_{1} + \varepsilon_{1} \geq \frac{1}{2} + \varepsilon_{I}\left(1 - \frac{1}{2s}\right) = \frac{1}{2} + \frac{1}{2} \cdot \frac{(2s - 1)^{2}}{12s^{2} + 4s - 1} - \kappa\left(1 - \frac{1}{2s}\right). \end{displaymath}
Either way, making $\kappa > 0$ small, we may choose $\varepsilon_{1}(s) > 0$ and $\sigma_{1}(s) > 1/2$ as in Theorem \ref{mainDiscrete}(a). This completes the proof.
\end{proof}

\section{Sets with additional structure}

The dimension estimates obtained up to now for  the projections of arbitrary sets $B\subset \mathbb{R}^3$  onto a non-degenerate family of lines in $\mathbb{R}^3$ are far from the optimal bound suggested by Conjecture \ref{dimensionConservation}. In this section, we restrict ourselves to a special class of sets $B$, for which we are able to prove stronger results -- and indeed resolve part of Conjecture \ref{dimensionConservation}. The section has two parts. In the first one, we introduce \emph{BLP sets}, a class of sets satisfying a strong structural hypothesis, and, in Theorem \ref{main2_BLP}, we obtain sharp dimension estimates for such sets and non-degenerate families of projections onto lines. In the second part, we demonstrate that self-similar sets without rotations are BLP sets. Combined with Theorem \ref{main2_BLP}, this fact yields a Marstrand type theorem for self similar sets: Theorem \ref{main2}.

\subsection{BLP sets}

We set off with two definitions.
\begin{definition}[BLP sets] A set $B \subset \R^{3}$ has the \emph{bi-Lipschitz property}, BLP in short, if for any plane $V \in G(3,2)$ and $\varepsilon > 0$ there exists a subset $B_{V,\varepsilon} \subset B$ such that
\begin{itemize}
\item $\Hd B_{V,\varepsilon} \geq \Hd \pi_{V}(B) - \varepsilon$, and
\item the restriction $\pi_{V}|_{B_{V,\varepsilon}} \colon B_{V,\varepsilon} \to V$ is bi-Lipschitz.
\end{itemize}
\end{definition}

\begin{definition} Let $\ell \in G(3,1)$. A set $B \subset \R^{3}$ \emph{stays non-tangentially off the line $\ell$} ($B \angle \ell$ for short) if there exists $0 < \alpha < 1$ such that
\begin{displaymath}
X(0,\ell,\alpha)\cap (B - B)=\emptyset,
\end{displaymath}
where
\begin{displaymath} X(y,\ell,\alpha):=\{x\in \mathbb{R}^3:\; d(x - y,\ell) < \alpha |x - y|\} \end{displaymath}
is a cone with opening angle $\alpha$ around $\ell$ centered at $y \in \R^{3}$.
\end{definition}

It will be useful to have various reformulations of this property at our disposal. We summarise them in the subsequent lemma, but omit the straightforward proof.

\begin{lemma}\label{l:equiv_snto}
Let $B\subset \mathbb{R}^3$ and $\ell \in G(3,1)$. The following properties are equivalent:
\begin{enumerate}
\item\label{i} $B \angle \ell$.
\item\label{ii} There exists $0 < \alpha < 1$ such that for all $y\in B$ we have $X(y,\ell,\alpha)\cap B =\emptyset$.
\item\label{iii} The projection $\pi_{V}|_{B}$ onto the plane $V=\ell^{\bot}$
is bi-Lipschitz with the constant $\alpha$ from the definition of $B \angle \ell$.
\end{enumerate}
\end{lemma}

Let us now return to the projection family $(\rho_{\theta})_{\theta \in U}$.
The point of the definitions above is here: if $B \subset \R^{3}$ is a BLP set, $\theta_{0} \in U$ and $\varepsilon > 0$, one may find a subset $B_{\theta_{0},\varepsilon} \subset B$ such that $\Hd B_{\theta_{0},\varepsilon} \geq \Hd \pi_{\theta_{0}}(B) - \varepsilon$ and $B_{\theta_{0},\varepsilon} \angle b_{\theta_{0}}$, where $b_{\theta_0} \in G(3,1)$ is the `bad line' spanned by the vector $\gamma(\theta_0) \times \dot{\gamma}(\theta_0)$ and $\pi_{\theta_{0}}$ is the projection onto $b_{\theta_{0}}^{\perp}$. The next proposition explains why this is useful:

\begin{proposition}\label{p:2}
Let $B \subset \mathbb{R}^3$ be a set such that $B\angle b_{\theta_0}$ for some $\theta_0\in U$. Then there exists an open interval $J \ni \theta_{0}$ such that the family $(\rho_{\theta}|_B)_{\theta \in J}$ is transversal in the sense of Peres and Schlag, see \cite[Definition 2.7]{PSc}.
\end{proposition}

\begin{remark}\label{PeresAndSchlag} It would be unnecessarily cumbersome to recount here the full details of Peres and Schlag's framework. So, for the benefit of readers unfamiliar with their definitions, we simply remark that Peres and Schlag's paper deals with \emph{generalised projections} -- parametrised families of continuous mappings from a compact space $\Omega$ to $\R^{k}$, satisfying certain properties. The essence of these properties is that they axiomatise those features of usual orthogonal projections which are needed for the proofs of the classical dimension conservation results of Marstrand and Mattila. Consequently, an analogous dimension conservation theorem holds for all families of generalised projections: in particular, the generalised projections preserve, for almost all parameters, the dimension of any at most $k$-dimensional Borel set in $\Omega$. We wish to use this fact in the proof of Theorem \ref{main2_BLP} below, so, prior to the proof, we need to check that a certain family of projections, namely $(\rho_{\theta}|_{B})_{\theta \in J}$ satisfies the axioms of the generalised projections. In this case, the task boils down to proving \eqref{transversality} below. \end{remark}

\begin{proof}[Proof of Proposition \ref{p:2}] Staying `non-tangentially off a line' is an open property in the following sense: if there exists $0<\alpha <1$ such that $(B-B)\cap X(0,b_{\theta_0},\alpha )=\emptyset$, as we assume, then
 $(B-B)\cap X(0,b_{\theta},\alpha/2)=\emptyset$ for $\theta$ in a small neighbourhood $J \subset U$ of $\theta_{0}$. Now, let $\theta \in J$, and consider the projection $\pi_{\theta}$ onto the plane $V_{\theta} = b_{\theta}^{\perp}$. According to Lemma \ref{l:equiv_snto}, the restriction $\pi_{\theta}|_{B}$ is bi-Lipschitz with constant $\alpha/2$, which means that
 \begin{displaymath}
\left[\left( (x-y)\cdot\gamma(\theta)\right)^2 + \left( (x-y)\cdot\frac{\dot{\gamma}(\theta)}{|\dot{\gamma}(\theta)|} \right)^2\right]^{1/2}=|\pi_{\theta}(x-y)| \geq \tfrac{\alpha}{2}|x-y|
\end{displaymath}
for all $x,y \in B$. Taking $J$ short enough, the quantity $|\dot{\gamma}(\theta)|$ is bounded from below by a constant $c > 0$ for $\theta \in J$. Thus, either
\begin{equation}\label{transversality}
\left|\rho_{\theta}\left(\tfrac{x-y}{|x-y|}\right)\right|\geq \tfrac{\alpha}{5}\quad \text{or}\quad \left|{\partial_{\theta}}\rho_{\theta}\left(\tfrac{x-y}{|x-y|}\right)\right|\geq \tfrac{c\alpha}{5}.\end{equation}
for all $x,y \in B$, $x \neq y$. This means that $J$ is an interval of transversality of order $\beta=0$ for the projection family $(\rho_{\theta}|_{B})_{\theta \in J}$, in the sense \cite[Definition 2.7]{PSc}.
\end{proof}

Now we are prepared to prove the analogue of Theorem \ref{main2} for BLP sets.

\begin{thm}\label{main2_BLP} Let $B \subset \R^{3}$ be a BLP set, and let $(\rho_{\theta})_{\theta \in U}$ be a non-degenerate family of projections in the sense of Definition \ref{projections}.
\begin{itemize}
\item[(a)] If $0 \leq \Hd B \leq 1$, then $\Hd \rho_{\theta}(B) = \Hd B$ almost surely.
\item[(b)] If $\Hd B > 1$, and additionally
\begin{equation}\label{additionalAssumption} \Pd \pi_{V}(B) = \Hd \pi_{V}(B) \end{equation}
for every plane $V \in G(3,2)$, then $\rho_{\theta}(B)$ has positive length almost surely.
\end{itemize}
\end{thm}

\begin{proof}
We start with (a). According to Lemma \ref{etaNonDeg}, the family of lines $(b_{\theta})_{\theta \in U}$ is a non-degenerate one. Using Proposition \ref{prop1}(b), we see that
\begin{equation}\label{form12} \Hd \pi_{\theta}(B) = \Hd B \end{equation}
for almost every $\theta \in U$, where $\pi_{\theta}$ refers to the projection onto the plane $V_{\theta} = b_{\theta}^{\perp}$. Let $\theta_0 \in U$ be one of the parameters for which \eqref{form12} holds, and fix $\varepsilon > 0$. Since $B$ is a BLP set, we may choose a subset $B_{\theta_{0},\varepsilon} \subset B$ such that $\Hd B_{\theta_{0},\varepsilon}\geq \Hd B - \varepsilon$ and $B_{\theta_{0},\varepsilon}\angle b_{\theta_{0}}$. Then, we infer from Proposition \ref{p:2} that there exists a small interval $J \subset U$ containing $\theta_{0}$ such that the projections $(\rho_{\theta})_{\theta \in J}$ restricted to $B_{\theta_{0},\varepsilon}$ are transversal. It follows from \cite[Theorem 2.8]{PSc} that
\begin{displaymath} \Hd \rho_{\theta}(B) \geq \Hd \rho_{\theta}(B_{\theta_{0},\varepsilon}) = \Hd B_{\theta_{0},\varepsilon} \geq \Hd B -\varepsilon \end{displaymath}
for almost every $\theta\in J$. Since \eqref{form12} holds almost surely, we can run the same argument for almost every $\theta_{0} \in U$, proving that $\Hd \rho_{\theta}(B) \geq \Hd B - \varepsilon$ for almost every $\theta \in U$. Letting $\varepsilon \to 0$ concludes the proof of part (a).

The proof of part (b) is similar, except that this time we resort to Theorem \ref{main1}(b) instead of Proposition \ref{prop1}(b). Namely, if $\Hd B > 1$, we infer from Theorem \ref{main1}(b) and the additional assumption \eqref{additionalAssumption} that
\begin{displaymath} \Hd \pi_{\theta}(B) = \Pd \pi_{\theta}(B) > 1 \end{displaymath}
for almost every $\theta \in U$. Then, fixing almost any $\theta_{0} \in U$ and using the BLP property, we find a subset $B_{\theta_{0}} \subset B$ such that $\Hd B_{\theta_{0}} > 1$ and $B_{\theta_{0}} \angle b_{\theta_{0}}$. The rest of the argument is the same a before, applying \cite[Theorem 2.8]{PSc} to the projections $\rho_{\theta}$, which are transversal restricted to the set $B_{\theta_{0}}$.
\end{proof}

Unfortunately, not all sets are BLP sets:

\begin{remark}\label{noStructure} It is easy to construct a compact set $K \subset \R^{3}$ with $\Hd K = 1$ such that
\begin{equation}\label{noStructureEq} \Hd \pi_{V}(K) = 0 \end{equation}
for a countable dense set of subspaces $V \in G(3,2)$. Any such set $K$ has the following property. Let $V_{0} \in G(3,2)$, and let $K_{0}$ be a subset of $K$ such that the restriction $\pi_{V_{0}}|_{K_{0}}$ is bi-Lipschitz. Then $\Hd K_{0} = 0$. Indeed, if $\pi_{V_{0}}|_{K_{0}}$ is bi-Lipschitz, then $\pi_{V}|_{K_{0}}$ is also bi-Lipschitz for all $2$-planes $V$ in a small $G(3,2)$-neighbourhood of $V_{0}$. This means that $\Hd \pi_{V}(K) \geq \Hd \pi_{V}(K_{0}) = \Hd K_{0}$ for all  $2$-planes $V$ in an open subset of $G(3,2)$, and now \eqref{noStructureEq} forces $\Hd K_{0} = 0$.
\end{remark}

\subsection{Self-similar sets} In this section, we prove that self-similar sets without rotations in $\R^{3}$ satisfy the assumptions of Theorem \ref{main2_BLP}. We start by setting some notation. Consider a collection $\{\psi_1,\ldots,\psi_q\}$ of contracting similitudes $\psi_i:\mathbb{R}^3 \to \mathbb{R}^3$. According to a result of Hutchinson \cite{Hu} there exists a unique nonempty compact set $K\subset \mathbb{R}^3$ satisfying $K=\bigcup_{i=1}^q \psi_i(K)$. Such sets $K$ are referred to as \emph{self-similar sets}. If the generating similitudes of $K$ have the form  $\psi_i(x)=r_ix + w_i$ with $0<r_i<1$ and $w_{i} \in \R^{3}$, we call $K$ a \emph{self-similar set without rotations}. The fact that the mappings $\psi_i$ do not involve rotations will be used to guarantee that the projection of $K$ to an arbitrary plane is again self-similar.


\begin{proposition}\label{p:3}
Every self-similar set in $\R^{3}$ without rotations is a BLP set.
\end{proposition}

Before presenting the proof, we recall some terminology from \cite{Or1}. Rescaling the translation vectors $w_{i}$ if necessary, we may assume that the similitudes $\psi_i$, $i\in \{1,\ldots,q\}$, map the ball $B(0,\frac{1}{2})$ into itself. Then, we set $\mathcal{B}_0=\{B(0,\frac{1}{2})\}$ and refer to the recursively defined family
\begin{displaymath}
\mathcal{B}_n:= \{\psi_j(B):\; B\in\mathcal{B}_{n-1},\, 1\leq j\leq q\}
\end{displaymath}
as the collection of \emph{generation $n$ balls of $K$ associated with $\{\psi_1,\ldots,\psi_q\}$}.

The subset $K_{V,\varepsilon}$ to be constructed in the proof of Proposition \ref{p:3} will be the attractor of a family of similitudes of the form $\{\psi_B:\; B\in \mathcal{G}\}$, where $\mathcal{G}$ is a suitably chosen collection of balls in $\bigcup_{m\in\mathbb{N}}\mathcal{B}_m$. Here, $\psi_B$ stands for a similitude of the form $\psi_{B} = \psi_{i_1}\circ \cdots \circ \psi_{i_n}$, mapping $B(0,\frac{1}{2})$ to $B=\psi_{i_1}\circ \cdots \psi_{i_n}(B(0,\frac{1}{2}))$. For a given $B \in \mathcal{B}_n$, the selection of $\psi_{i_1},\ldots,\psi_{i_n}$ may not be unique, but then any choice is equally good for us. Observe that, for an arbitrary collection of balls $\mathcal{G}\subseteq \bigcup_{m\in\mathbb{N}}\mathcal{B}_m$, the associated attractor is a subset of $K$. Also, since $\mathcal{B}_0$ was defined to consist of a single ball of diameter one, $\psi_B$ has contraction ratio $\mathrm{diam}(B)$.

If $\{r_1,\ldots,r_q\}$ are the contraction ratios of an IFS $\{\psi_1,\ldots,\psi_q\}$, then the \emph{similarity dimension} of the associated attractor $K$ is defined as the unique number $s\geq 0$, which solves the equation
\begin{displaymath}
\sum_{j=1}^q r_j^s =1.
\end{displaymath}
It is well known, see \cite{Hu}, that $s=\Hd K$, provided that $K$ exhibits a sufficient degree of separation. One such condition is the \emph{very strong separation condition}, which, by definition, requires the generation $1$ balls of $K$ to be disjoint. It is a stronger requirement than the \emph{open set condition} commonly used in literature, but will be very convenient in the proof of Proposition \ref{p:3}.

\begin{proof}[Proof of Proposition \ref{p:3}] Let $V\in G(3,2)$ and $\varepsilon > 0$ be arbitrary. The assumption that the similitudes $\psi_1,\ldots,\psi_q$ generating the self-similar set $K$ contain no rotations ensures that the set $\pi_{V}(K)$ is again self-similar. It is a subset of $V$, or, under the customary identification, a subset of $\mathbb{R}^2$, given by the IFS
\begin{displaymath}
\{\psi_{1,V},\ldots,\psi_{q,V}\}\quad \text{with }\psi_{j,V}:\mathbb{R}^2 \to \mathbb{R}^2,\;\psi_{j,V}(x)=r_j x + \pi_V(w_j).
\end{displaymath}
The corresponding collection of generation $n$ balls will be denoted by $\mathcal{B}_{n,V}$. Observe that the ball $B(0,\frac{1}{2})$ in $\mathbb{R}^3$ is projected to the ball $B(0,\tfrac{1}{2})$ in $\mathbb{R}^2$, and hence $\mathcal{B}_{n,V}$ comprises precisely the projections of the balls in $\mathcal{B}_n$.

 According to Lemma 3.4 in \cite{Or1}, we can  for every $\varepsilon>0$ choose a self-similar set $K^V\subset \pi_V(K)$ (depending on $\varepsilon$) with $\Hd K^V\geq \Hd \pi_V(K)-\varepsilon$ satisfying the very strong separation condition. In fact, the proof in \cite{Or1} provides an IFS, which generates the set $K^V$ and for which the generation $1$ balls are a subcollection $\mathcal{B}_1^V$ of disjoint balls in $\mathcal{B}_{n,V}$, for some large $n \in \N$. Moreover, we have
 \begin{equation}\label{form13} \sum_{B\in \mathcal{B}_1^V} \mathrm{diam}(B)^{s}=1 \end{equation}
 with $s=\Hd K^V$.
Each ball $B\in \mathcal{B}_1^V$ is the image of a ball in $\mathcal{B}_n$ under the projection $\pi_V$. There might be several such balls in $\mathcal{B}_{n}$, but we just pick one of them. We denote by $\mathcal{G}_{1}$ the collection of balls in $\mathbb{R}^3$ obtained in this way. Since the balls in $\mathcal{B}_1^V$ are disjoint, the balls in $\mathcal{G}_1$ are contained in disjoint well-separated tubes perpendicular to $V$ (and thus parallel to $V^{\bot}$). Also, \eqref{form13} implies that
\begin{equation}\label{eq:eq_dims}
 \sum_{B\in \mathcal{G}_1} \mathrm{diam}(B)^{s}=1
 \end{equation}
with $s=\Hd K^V$.
The set $K_{V,\varepsilon} \subset K$, whose existence is claimed in the statement of the proposition, is obtained as the attractor of the IFS $\{\psi_B:\; B\in\mathcal{G}_1\}$. In other words, the balls in $\mathcal{G}_1$ form the generation $1$ balls of $K_{V,\varepsilon}$. By \eqref{eq:eq_dims} and the strong separation condition, we have
\begin{displaymath}
\Hd K_{V,\varepsilon} = \Hd K^V \geq \Hd \pi_V(K)-\varepsilon.
\end{displaymath}

It remains to be established that the restriction of $\pi_V$ to $K_{V,\varepsilon}$ is bi-Lipschitz. To this end, we use the equivalent characterisation of this property in terms of cones as stated in Lemma \ref{l:equiv_snto}. So far, we know that distinct balls in $\mathcal{G}_1$ are contained in disjoint closed tubes in direction $V^{\bot}$. This allows us to find $\alpha>0$ so that
\begin{equation}\label{eq:cone_sep_1}
B \cap X(y,V^{\bot},\alpha) =\emptyset \quad \text{for all }y\in B',
\end{equation}
whenever $B$ and $B'$ are distinct balls in $\mathcal{G}_1$. Then, it is a consequence of self-similarity that \eqref{eq:cone_sep_1} holds \textbf{with the same constant $\alpha$} for distinct generation $n$ balls of $K_{V,\varepsilon}$, for any $n \in \N$. This is the content of the following lemma, a counterpart of which for sets in the plane is \cite[Proposition 4.14]{Or1}. Since the proof in higher dimensions is completely analogous, we omit it here.

\begin{lemma}\label{l:cone_sep}
Let $n\in\mathbb{N}$ be arbitrary and denote by $\mathcal{G}_n$  the generation $n$ balls of $K_{V,\varepsilon}$. Then, whenever $B$ and $B'$ are distinct balls in $\mathcal{G}_n$, we have
\begin{equation*}
B \cap X(y,V^{\bot},\alpha) =\emptyset \quad \text{for all }y\in B'.
\end{equation*}
\end{lemma}

Consequently, the restriction of $\pi_{V}$ to $K_{V,\varepsilon}$ is bi-Lipschitz with constant $\alpha$, and the proof of the proposition is complete. \end{proof}

Finally, Theorem \ref{main2} follows by combining Theorem \ref{main2_BLP} with the BLP property of self-similar sets established in Proposition \ref{p:3}.

\section{Further results}\label{furtherResults}

This section contains further results concerning the projections onto a non-degenerate family of lines. It consists of two parts that, in specific situations, provide additional information to the dimension bounds obtained in Theorem \ref{main1}(a).
In the first part, we consider a special non-degenerate family of lines, namely those foliating the surface of a cone, and the result obtained only applies to sets of a certain product form. In the second part, we again return to arbitrary non-degenerate families of lines and prove an explicit dimension bound for the associated projections as stated in Remark \ref{bestBounds}.

\subsection{Product sets and projections onto lines on a cone}

Let $K = K_1 \times K_2$ be a product set in $\mathbb{R}^3$ with $K_1 \subset \mathbb{R}^2$ and $K_2 \subset \mathbb{R}$, and consider the curve $\gamma \colon (0,2\pi) \to S(0,\sqrt{2})$ given by
\begin{displaymath}
\gamma(\theta)= (\cos(\theta), \sin(\theta),1).
\end{displaymath}
Then the lines $\ell_{\theta} := \spa(\gamma(\theta))$, $\theta \in (0,2\pi)$, foliate the surface of a vertical cone in $\R^{3}$, and the projections of $K$ under $\rho_{\theta}(x) := \gamma(\theta) \cdot x$ have a particularly simple form:
\begin{equation}\label{productProj} \rho_{\theta}(K) = \rho_{\theta}(K_{1} \times K_{2}) = \wp_{\theta}(K_{1}) + K_{2}, \end{equation}
where $\wp_{\theta} \colon \R^{2} \to \R$ is the planar projection $\wp_{\theta}(x,y) = x \cos \theta + y \sin \theta$. It is easy to verify that the curve $\gamma$ (normalised by a constant) satisfies the non-degeneracy condition \eqref{eq:curve_cond}, so Theorem \ref{main2} holds for the projections $\rho_{\theta}$. Applying part (b) to the $3$-fold product of an equicontractive self-similar set in $\R$ (which is a self-similar set in $\R^{3}$) and recalling \eqref{productProj} yields Corollary \ref{cor1}.

As the first `further result', we prove a variant of Theorem \ref{main1}(a) for product sets $K = K_{1} \times K_{2}$ and the special family of projections $\rho_{\theta}$ defined above.

\begin{proposition}
Let $K=K_1\times K_2 \subset \R^{3}$, where $K_1\subset \mathbb{R}^2$, $K_2 \subset \mathbb{R}$ are analytic sets. Then $\Hd \rho_{\theta}(K)\geq \mathrm{min}\{\tfrac{1}{2},\Hd K_1\} + \Hd K_2$ for almost every $\theta\in (0,2\pi)$.
\end{proposition}

\begin{proof} Let $\mu_{1}$ and $\mu_{2}$ be positive Borel measures supported on $K_1$ and $K_2$, respectively, such that $I_{t_{1}}(\mu_{1}) < \infty$ for some $0 < t_{1} < \min\{\Hd K_{1}, 1/2\}$ and $I_{t_{2}}(\mu_{2}) < \infty$ for some $0 < t_{2} < \Hd K_{2}$. Then, with $\mu=\mu_1\times \mu_2$, we have
\begin{align*}
\int_{0}^{2\pi} I_{t_1+t_2}(\rho_{\theta\sharp}\mu) \, d\theta&= \int_{0}^{2\pi}\left(\int |\widehat{\rho_{\theta\sharp}\mu}(r)|^2|r|^{t_1+t_2-1}dr\right)d\theta\\
&\sim \int_{0}^{2\pi}\left(\int |\hat \mu(r\gamma(\theta))|^2 |r|^{t_1+t_2-1}dr\right)d\theta\\
&= \int |\hat \mu_2(r)|^2 \left(\int_{0}^{2\pi}|\hat \mu_1(r \cos \theta,r \sin \theta)|^{2} \, d\theta\right)|r|^{t_1+t_2-1}dr.
\end{align*}
The inner integral is, by definition, the \emph{spherical average} $\sigma(\mu_{1})(|r|)$ of $\mu_1$ and an estimate of P.\ Mattila, see \cite[Theorem 3.8]{Mat2}, yields
\begin{displaymath}
\sigma(\mu_{1})\left(|r|\right) \lesssim |r|^{-t_1}I_{t_1}(\mu_1).
\end{displaymath}
Here we needed the assumption $t_1<1/2$, which guarantees that $t_1$ is within the range where the results from \cite{Mat2} apply. We may now conclude that

\begin{displaymath}
\int_{0}^{2\pi} I_{t_1+t_2}(\rho_{\theta\sharp}\mu) \, d\theta \lesssim I_{t_1}(\mu_1)\int |\hat \mu_2(r)|^2 |r|^{t_2-1} \, dr \sim I_{t_1}(\mu_1)I_{t_2}(\mu_2)<\infty,
\end{displaymath}
and thus $I_{t_1+t_2}(\rho_{\theta\sharp}\mu)<\infty$ for almost every $\theta$. This implies that
\begin{displaymath} \Hd \rho_{\theta}(K_1 \times K_2)\geq t_1+t_2 \end{displaymath}
for almost every $\theta \in (0,2\pi)$, and the proposition follows.
\end{proof}

Before moving on to other topics, we remark that, in light of \eqref{productProj}, the following conjecture is a weaker variant of Conjecture \ref{dimensionConservation}:
\begin{conjecture} Let $K_{1} \subset \R^{2}$ and $K_{2} \subset \R$ be analytic sets satisfying $\Hd K_{1} + \Hd K_{2} \leq 1$. Then
\begin{displaymath}
\Hd (\wp_{\theta}(K_1)+K_2) \geq \Hd K_{1} + \Hd K_{2} \end{displaymath}
for almost every $\theta\in (0,2\pi)$.
\end{conjecture}

\subsection{Another lower bound for general sets}

In this section, we consider the general one-dimensional family of projections $(\rho_{\theta})_{\theta \in U}$.
\begin{proposition}\label{s/2} If $K \subset \R^{3}$ is an analytic set with $0 \leq \Hd K \leq 2$, then $\Pd \rho_{\theta}(K) \geq \Hd K/2$ for almost every $\theta \in U$.
\end{proposition}

The proposition starts improving on the lower bound for $\sigma_{1}(s) > 1/2$ from Remark \ref{bestBounds} when $\Hd K = s \approx 1.077$.

\begin{proof}[Proof of Proposition \ref{s/2}] Write $\Hd K =: s$ and assume $0 < s \leq 2$. To reach a contradiction, suppose that there is a set $E \subset U$ with positive length such that $\Pd \rho_{\theta}(K) < s/2$ for every $\theta \in E$. Find two distinct Lebesgue points $\theta_{1},\theta_{2} \in E$ such that
\begin{displaymath} (\gamma(\theta_{1}) \times \dot{\gamma}(\theta_{1})) \cdot \dot{\gamma}(\theta_{2}) \neq 0. \end{displaymath}
Such points are given by the same argument as we used to obtain \eqref{form1}. Next, use continuity to find short open neighbourhoods $I,J \subset U$ of $\theta_{1}$ and $\theta_{2}$ such that
\begin{equation}\label{form11} |(\gamma(\theta_{I}) \times \dot{\gamma}(\theta_{I})) \cdot \dot{\gamma}(\theta_{J})| \geq c > 0 \end{equation}
for all $(\theta_{I},\theta_{J}) \in I \times J \subset \R^{2}$. Then, consider the two-parameter family of projections $\Pi_{(\theta_{I},\theta_{J})} \colon \R^{3} \to \R^{2}$, $(\theta_{I},\theta_{J}) \in I \times J$, given by
\begin{displaymath} \Pi_{(\theta_{I},\theta_{J})}(x) := (\rho_{\theta_{I}}(x),\rho_{\theta_{J}}(x)) = (\gamma(\theta_{I}) \cdot x, \gamma(\theta_{J}) \cdot x). \end{displaymath}
Using \eqref{form11}, one may check that this is a family of generalised projections satisfying the framework of Peres and Schlag, see \cite[Definitions 7.1 and 7.2]{PSc}. Indeed, if $x \in \R^{3}$ is a unit vector such that, simultaneously, $\Pi_{(\theta_{I},\theta_{J})}(x) = 0$ and
\begin{equation}\label{transversality2} 0 = \det[D\Pi_{(\theta_{I},\theta_{J})}(x)(D\Pi_{(\theta_{I},\theta_{J})}(x))^{T}] = (\dot{\gamma}(\theta_{I}) \cdot x)^{2} + (\dot{\gamma}(\theta_{J}) \cdot x)^{2}, \end{equation}
then $x$ is perpendicular to both planes $\spa(\{\gamma(\theta_{I}),\gamma(\theta_{J})\})$ and $\spa(\{\dot{\gamma}(\theta_{I}),\dot{\gamma}(\theta_{J})\})$, which implies that
\begin{displaymath} 0 = \gamma(\theta_{I}) \cdot (\dot{\gamma}(\theta_{I}) \times \dot{\gamma}(\theta_{J})) = (\gamma(\theta_{I}) \times \dot{\gamma}(\theta_{I})) \cdot \dot{\gamma}(\theta_{J}), \end{displaymath}
violating \eqref{form11}. Since we have not properly introduced the "generalised projections" framework of Peres and Schlag (see Remark \ref{PeresAndSchlag}), we can only state that checking the requirements in \cite[Definitions 7.1 and 7.2]{PSc} amounts precisely to verifying that $\Pi_{(\theta_{I},\theta_{J})}(x) = 0$ and the equation \eqref{transversality2} cannot hold simultaneously -- and this we have just done.

Now \cite[Theorem 7.3]{PSc} implies that $\Hd \Pi_{(\theta_{I},\theta_{J})}(K) = s$ for almost every pair $(\theta_{I},\theta_{J}) \in I \times J$. On the other hand,
\begin{displaymath} \Pi_{(\theta_{I},\theta_{J})}(K) \subset \rho_{\theta_{I}}(K) \times \rho_{\theta_{J}}(K), \end{displaymath}
so we obtain the estimate
\begin{displaymath} \Pd \rho_{\theta_{I}}(K) + \Pd \rho_{\theta_{J}}(K) \geq \Hd \Pi_{(\theta_{I},\theta_{J})}(K) = s \end{displaymath}
for almost every pair $(\theta_{I},\theta_{J}) \in I \times J$. For such pairs $(\theta_{I},\theta_{J})$, we have either $\Pd \rho_{\theta_{I}}(K) \geq s/2$ or $\Pd \rho_{\theta_{J}}(K) \geq s/2$, which means that
\begin{align*} |I||J| & = |I \times J| = |\{(\theta_{I},\theta_{J}) : \Pd \rho_{\theta_{I}}(K) \geq \tfrac{s}{2} \text{ or } \Pd \rho_{\theta_{J}}(K) \geq \tfrac{s}{2}\}|\\
& \leq \left|\left\{\theta_{I} \in I : \Pd \rho_{\theta_{I}}(K) \geq \tfrac{s}{2}\right\}\right||J| + |I|\left|\left\{\theta_{J} \in J : \Pd \rho_{\theta_{J}}(K) \geq \tfrac{s}{2}\right\}\right|. \end{align*}
We may conclude that either
\begin{displaymath} \left|\left\{\theta_{I} \in I : \Pd \rho_{\theta_{I}}(K) \geq \tfrac{s}{2}\right\}\right| \geq \frac{|I|}{2} \quad \text{ or } \quad \left|\left\{\theta_{J} \in J : \Pd \rho_{\theta_{J}}(K) \geq \tfrac{s}{2}\right\}\right| \geq \frac{|J|}{2}. \end{displaymath}
However, since $\theta_{1}$ and $\theta_{2}$ were Lebesgue points of $E$, neither option is possible if $I$ and $J$ were chosen short enough to begin with. This contradiction completes the proof.
\end{proof}

\appendix

\section{A discrete version of Frostman's lemma}\label{appFrostman}

In this section, we prove Lemma \ref{frostman}. Let us recall the statement:

\begin{proposition} Let $\delta > 0$, and let $B \subset \R^{3}$ be a set with $\cH^{s}_{\infty}(B) =: \kappa > 0$. Then, there exists a $(\delta,s)$-set $P \subset B$ with cardinality $|P| \gtrsim \kappa \cdot \delta^{-s}$.
\end{proposition}

\begin{proof} Without loss of generality, assume that $\delta = 2^{-k}$ for some $k \in \N$ and $B \subset [0,1]^{3}$. Denote by $\cD_{k}$ the dyadic cubes in $\R^{3}$ of side-length $2^{-k}$. First, find all the dyadic cubes $Q^{k} \in \cD_{k}$ which intersect $B$, and choose a single point $x \in B \cap Q^{k}$ for each $Q^{k}$. The finite set so obtained is denoted by $P_{0}$. Next, modify $P_{0}$ as follows. Consider the cubes in $\cD_{k - 1}$. If one of these, say $Q^{k - 1}$, satisfies
\begin{displaymath} |P_{0} \cap Q^{k - 1}| > \left(\frac{d(Q^{k - 1})}{\delta}\right)^{s}, \end{displaymath}
remove points from $P_{0} \cap Q^{k - 1}$, until the reduced set $P_{0}'$ satisfies
\begin{displaymath} \frac{1}{2}\left(\frac{d(Q^{k - 1})}{\delta}\right)^{s} \leq |P_{0}' \cap Q^{k - 1}| \leq \left(\frac{d(Q^{k - 1})}{\delta}\right)^{s}. \end{displaymath}
Repeat this for all cubes $Q^{k - 1} \in \cD_{k - 1}$ to obtain $P_{1}$. Then, repeat the procedure at all dyadic scales up from $\delta$, one scale at a time: whenever $P_{j}$ has been defined, and there is a cube $Q^{k - j - 1} \in \cD_{k - j - 1}$ such that
\begin{displaymath} |P_{j} \cap Q^{k - j - 1}| > \left(\frac{d(Q^{k - j - 1})}{\delta}\right)^{s}, \end{displaymath}
remove points from $P_{j} \cap Q^{k - j - 1}$, until the reduced set $P_{j}'$ satisfies
\begin{equation}\label{appForm1} \frac{1}{2}\left(\frac{d(Q^{k - j - 1})}{\delta}\right)^{s} \leq |P_{j}' \cap Q^{k - j - 1}| \leq \left(\frac{d(Q^{k - j - 1})}{\delta}\right)^{s}. \end{equation}
Stop the process when the remaining set of points, denoted by $P$, is entirely contained in some dyadic cube $Q_{0} \subset [0,1]^{3}$. Now, we claim that for every point $x \in P_{0}$ there exists a unique maximal dyadic cube $Q_{x} \subset Q_{0}$ such that $\ell(Q_{x}) \geq \delta$ and
\begin{equation}\label{appForm2} |P \cap Q_{x}| \geq \frac{1}{2}\left(\frac{d(Q_{x})}{\delta}\right)^{s}. \end{equation}
We only need to show that there exists \textbf{at least one} cube $Q_{x} \ni x$ satisfying \eqref{appForm2}; the rest follows automatically from the dyadic structure. If $x \in P$, we have \eqref{appForm2} for the dyadic cube $Q_{x} \in \cD_{k}$ containing $x$. On the other hand, if $x \in P_{0} \setminus P$, the point $x$ was deleted from $P_{0}$ at some stage. Then, it makes sense to define $Q_{x}$ as the dyadic cube containing $x$, where the `last deletion of points' occurred. If this happened while defining $P_{j + 1}$, we have \eqref{appForm1} with $Q_{k - j - 1} = Q_{x}$. But since this was the last cube containing $x$, where \textbf{any} deletion of points occurred, we see that that $P_{j}' \cap Q_{x} = P \cap Q_{x}$. This gives \eqref{appForm2}.

Now, observe that the cubes $\{Q_{x} : x \in P_{0}\}$,
\begin{itemize}
\item cover $B$, because they cover every cube in $\cD_{k}$ containing a point in $P_{0}$, and these cubes cover $B$,
\item are disjoint, hence partition the set $P$.
\end{itemize}
These facts and \eqref{appForm2} yield the lower bound
\begin{displaymath} |P| = \sum |P \cap Q_{x}| \gtrsim \delta^{-s}\sum d(Q_{x})^{s} \geq \kappa \cdot \delta^{-s}. \end{displaymath}
It remains to prove that $P$ is a $(\delta,s)$-set. For dyadic cubes $Q \in \cD_{l}$ with $l \leq k$ it follows immediately from the construction of $P$, in particular the right hand side of \eqref{appForm1}, that
\begin{displaymath} |P \cap Q| \leq \left(\frac{d(Q)}{\delta}\right)^{s}. \end{displaymath}
The statement for balls $B \subset \R^{3}$ with $d(B) \geq \delta$ follows by observing that any such ball can be covered by $\sim 1$ dyadic cubes of diameter $\sim d(B)$.
\end{proof}

\section{Auxiliary results for curves}\label{appendix_b}

In this section, we prove Lemma \ref{l:no_zeros}, Lemma  \ref{CICJ} and Lemma \ref{etaNonDeg}.

\begin{proof}[Proof of Lemma \ref{l:no_zeros}]
Consider the function
\begin{displaymath}
\Pi:[0,1]\times S^2 \to \mathbb{R},\quad \Pi(\theta,x):=\rho_{\theta}(x)= \gamma(\theta)\cdot x,
\end{displaymath}
and let $\delta > 0$ be a constant such that
\begin{equation}\label{c}
\max\left\{|\Pi(\theta,x)|,
\left|{\partial_{\theta} \Pi}(\theta,x)\right|,
\left|{\partial^2_{\theta} \Pi}(\theta,x)\right|\right\}\geq \delta, \quad (\theta,x)\in [0,1]\times S^2.
\end{equation}
Then, find $\varepsilon >0$ so that for all $(\theta,x),(\theta',x)\in [0,1]\times S^2$ with $|\theta - \theta'|<\varepsilon$, we have
\begin{equation}\label{eq:cont}
\max\left\{|\Pi(\theta,x)-\Pi(\theta',x)|,
\left|{\partial_{\theta} \Pi}(\theta,x)-{\partial_{\theta} \Pi}(\theta',x)\right|,
\left|{\partial^2_{\theta} \Pi}(\theta,x)-{\partial^2_{\theta} \Pi}(\theta',x)\right|\right\}<\delta.
\end{equation}
We claim that the statement of the lemma holds for this choice of $\varepsilon$. Fix $x\in S^2$ and let $I\subset [0,1]$ be an interval of length $\varepsilon$. To reach a contradiction, assume that there exist distinct points $\theta_1,\theta_2,\theta_3\in I$ such that
\begin{displaymath}
\Pi(\theta_1,x)=\Pi(\theta_2,x)=\Pi(\theta_3,x)=0.
\end{displaymath}
Applying Rolle's theorem to the function $\theta \mapsto \Pi(\theta,x)$, we conclude that there are at least two points in $I$ where also the derivative $\partial_{\theta}\Pi(\cdot,x)$ vanishes, and, by another application of Rolle's theorem, we find a point in $I$ where also $\partial^2_{\theta}\Pi(\cdot,x)$ is zero. From \eqref{eq:cont} it follows that
\begin{displaymath}
\max\left\{|\Pi(\theta,x)|,
\left|{\partial_{\theta} \Pi}(\theta,x)\right|,
\left|{\partial^2_{\theta} \Pi}(\theta,x)\right|\right\}<\delta
\end{displaymath}
for all $\theta\in I$, which contradicts \eqref{c}. \end{proof}

\begin{proof}[Proof of Lemma \ref{CICJ}]
Our goal is to find $\varepsilon_1,\varepsilon_2>0$ and $L < 1$ such that
\begin{equation}\label{eq:goal_C_I_C_J}
|\rho_{\theta}(x-y)|\leq L |x-y|\quad \text{for all }x,y \in C_I \text{ and }\theta\in C_J,
\end{equation}
where  $I = [\theta_{1} - \varepsilon_{1}, \theta_{1} + \varepsilon_{1}]$ and $J = [\theta_{2} - \varepsilon_{2}, \theta_{2} + \varepsilon_{2}]$. Elements in $C_I$ are of the form $x=r_{x}\gamma(\theta_{x})$ with $r_{x} \in \R$ and $\theta_{x} \in I$. As we will explain now, it is enough to verify \eqref{eq:goal_C_I_C_J} for pairs $x = r_{x}\gamma(\theta_{x}) \in C_{I}$ and $y = r_{y}\gamma(\theta_{y}) \in C_{I}$ with $r_{x},r_{y} \geq 0$. Clearly, if \eqref{eq:goal_C_I_C_J} holds for all such pairs $x,y$ then it also holds for pairs $x,y$ with $r_x,r_{y} \leq 0$. In case $r_{x}$ and $r_{y}$ have opposite signs, \eqref{eq:goal_C_I_C_J} will be valid with some constants $L'\in [L,1)$ and $\varepsilon_1' < \varepsilon_{1}$. The precise condition on $\varepsilon_1'>0$ is that
\begin{equation}\label{eq:cond_L}
L^2 < \min_{\theta_x,\theta_y\in I'} \gamma(\theta_x)\cdot \gamma(\theta_y)\quad\text{with} \quad I':= [\theta_1-\varepsilon_1',\theta_1+\varepsilon_1'].
\end{equation}
This can be achieved by the continuity of $\gamma$, since $\gamma(\theta_{1}) \cdot \gamma(\theta_{1}) = 1$ and $L < 1$. Now, fix $x = r_{x}\gamma(\theta_{x}) \in C_{I}$ and $y = r_{y}\gamma(\theta_{y}) \in C_{I}$ with $r_{x}r_{y} \leq 0$. Then, assuming \eqref{eq:goal_C_I_C_J} for points on $C_{I}$ with the same sign, we have
\begin{displaymath}
|\rho_{\theta}(x-y)|\leq |\rho_{\theta}(x)|+|\rho_{\theta}(y)|\leq L(|x|+|y|)\leq dL |x-y|= L'|x-y|,
\end{displaymath}
where
\begin{displaymath}
d = \left(\min_{\theta_x,\theta_y \in I'}\gamma(\theta_x)\cdot \gamma(\theta_y)\right)^{-1/2}\geq 1
\quad
\text{and}
\quad
L'= d L < 1
\end{displaymath}
by \eqref{eq:cond_L}. The inequality $|x|+|y|\leq d |x-y|$ follows from
\begin{displaymath}
\left(|x|+|y|\right)^2= r_x^2 +r_y^2 -2r_xr_y \leq d^2(r_x^2+r_y^2-2r_xr_y (\gamma(\theta_x)\cdot \gamma(\theta_y)) = d^2 |x-y|^2.
\end{displaymath}

It remains to prove \eqref{eq:goal_C_I_C_J} for points $x=r_x \gamma(\theta_x)$ and $y=r_y \gamma(\theta_y)$ with $r_x,r_y \geq 0$ and $\theta_{x},\theta_{y} \in I$. Without loss of generality we assume that $r_x \leq r_y$. The differentiability of $\gamma$ at $\theta_y$ yields
\begin{align*} |\rho_{\theta}(x-y)|&=|[r_{x}\gamma(\theta_{x}) - r_{y}\gamma(\theta_{y})] \cdot \gamma(\theta)|\\
& \leq |\left[r_{x}\dot{\gamma}(\theta_{y})(\theta_{x} - \theta_{y}) + (r_{x} - r_{y})\gamma(\theta_{y})\right] \cdot \gamma(\theta)| + r_x o(|\theta_{x} - \theta_{y}|).
\end{align*}
Exploiting the assumption $\gamma(\theta_2)\notin \mathrm{span}(\{\gamma(\theta_1),\dot{\gamma}(\theta_1)\})$, we can find constants $L_{0} < 1$ and $\varepsilon_1,\varepsilon_2>0$ such that
\begin{align*}
|\left[r_{x}\dot{\gamma}(\theta_{y})(\theta_{x} - \theta_{y}) + (r_{x} - r_{y})\gamma(\theta_{y})\right] \cdot \gamma(\theta)|\leq L_{0} \sqrt{r_x^2 |\dot{\gamma}(\theta_y)|^2(\theta_x-\theta_y)^2 + (r_x-r_y)^2}
\end{align*}
for all $\theta_x,\theta_y \in  [\theta_1-\varepsilon_1,\theta_1+\varepsilon_1] =: I$ and $\theta\in [\theta_2-\varepsilon_2,\theta_2+\varepsilon_2] =: J$. Given $\varepsilon > 0$, we can make $\varepsilon_{1}$ smaller still to ensure that the inequality
\begin{displaymath}
|\dot{\gamma}(\theta_y)||\theta_x-\theta_y|\leq (1+\varepsilon) |\gamma(\theta_x)-\gamma(\theta_y)|
\end{displaymath}
holds whenever $|\theta_x-\theta_y|\leq 2\varepsilon_1$.
Choosing $\varepsilon > 0$ small enough, inserting this estimate to the upper bound for $|\rho_{\theta}(x-y)|$, and comparing the result with
\begin{displaymath}
|x-y|=\sqrt{r_x^2 + r_y^2 -2 r_x r_y \gamma(\theta_x)\cdot\gamma(\theta_y)}=\sqrt{r_xr_y |\gamma(\theta_x)-\gamma(\theta_y)|^2 + (r_x-r_y)^2},
\end{displaymath}
we see that $|\rho_{\theta}(x-y)|\leq L|x-y|$ for some $L \in (L_{0},1)$. We also need to know that the bounds implicit in $o(|\theta_{x} - \theta_{y}|)$ can be chosen small in a manner depending only on $\varepsilon_{1}$, but this follows from the $\mathcal{C}^2$ regularity of $\gamma$.
\end{proof}

\begin{proof}[Proof of Lemma \ref{etaNonDeg}]
In order to establish
\eqref{eq:non_deg_lambda}
for all $\theta \in U$, it is sufficient to show
\begin{equation}\label{eq:goal_eta_2}
\ddot{\eta}(\theta)\cdot (\eta(\theta)\times \dot{\eta}(\theta))\neq 0,\quad\text{for all }\theta \in U.
\end{equation}
This condition means precisely that $\eta(\theta),\dot{\eta}(\theta)$ and $\ddot{\eta}(\theta)$ are all of positive length, ${\eta}(\theta)$ and $\dot{\eta}(\theta)$ are not parallel and hence span a plane, and this plane does not contain $\ddot{\eta}(\theta)$. In order to prove \eqref{eq:goal_eta_2},
 we first evaluate
\begin{align*}
\eta = \frac{\gamma \times \dot{\gamma}}{|\gamma \times \dot{\gamma}|}=\frac{1}{|\dot{\gamma}|}\gamma \times \dot{\gamma}, \qquad \dot{\eta} = \left(\frac{1}{|\dot{\gamma}|}\right)'\gamma \times \dot{\gamma} + \frac{1}{|\dot{\gamma}|}\gamma \times \ddot{\gamma} \end{align*}
and
\begin{align*} \ddot{\eta} = \left(\frac{1}{|\dot{\gamma}|}\right)''\gamma \times \dot{\gamma}+ 2\left(\frac{1}{|\dot{\gamma}|}\right)'\gamma \times \ddot{\gamma}+\frac{1}{|\dot{\gamma}|}\dot{\gamma} \times \ddot{\gamma}+\frac{1}{|\dot{\gamma}|}\gamma \times \dddot{\gamma}.
\end{align*}
Then,
\begin{displaymath}
\eta \times \dot{\eta}=\frac{1}{|\dot{\gamma}|^2}(\gamma \times \dot{\gamma})\times (\gamma \times \ddot{\gamma})=\frac{1}{|\dot{\gamma}|^2}\left(\gamma \cdot \left(\dot{\gamma}\times \ddot{\gamma}\right)\right) \gamma.
\end{displaymath}
Finally,
\begin{displaymath}
\ddot{\eta}\cdot (\eta \times \dot{\eta})= \frac{1}{|\dot{\gamma}|^3}\left(\gamma \cdot (\dot{\gamma}\times \ddot{\gamma})\right)^2,
\end{displaymath}
which is non-vanishing,  due to condition \eqref{eq:curve_cond} for the curve $\gamma$. This concludes the proof of the lemma.
\end{proof}

\end{document}